\documentclass[11pt]{article}
\usepackage{graphicx}
\usepackage{amsthm, amsmath, amssymb, tikz, bm}
\usepackage{mathtools, intcalc}
\usepackage{xifthen}
\usepackage{comment}
\usepackage[colorlinks=true,
linkcolor=blue,citecolor=blue,
urlcolor=blue]{hyperref}

\usepackage[margin=1.2in]{geometry} 
\usepackage{tikz}
\usetikzlibrary{automata, positioning}

\usepackage[shortlabels]{enumitem}
\usepackage{todonotes}
\usetikzlibrary{calc,shapes, backgrounds}
\allowdisplaybreaks

\usepackage[utf8]{inputenc}
\usepackage[T1,T2A]{fontenc}
\usepackage[russian,english]{babel}

\newtheorem{lemma}{Lemma}
\newtheorem{theorem}{Theorem}
\newtheorem{corollary}{Corollary}

\newtheorem{proposition}{Proposition}

\newtheorem{fact}{Fact}

\usepackage[normalem]{ulem}

\title{Independent Bondage Number in Graphs under Girth Constraints}

\author{
  E.G.K.M. Gamlath\thanks{University of Mississippi, University, MS 38677 (\texttt{egkmgamlath@gmail.com}).}
  \thanks{Franklin Pierce University, Rindge, NH 03461.}
  \and
  Andrew Pham\thanks{Alabama A\&M University, 4900 Meridian St N, Huntsville, AL 35811.}
  \and
  Bing Wei\footnotemark[1]
}
\begin{document}

\maketitle
\begin{abstract}
Given a finite, simple graph $G$, the independent bondage number of $G$ is the minimum size of an edge set such that its deletion results in a graph with strictly larger independent domination number than that of $G$. While the bondage number of graphs under girth constraints has been studied, very few results have yet been established for the independent bondage number. In this study, we establish upper bounds on the independent bondage number of planar graphs under given girth constraints, extending results on the bondage number by Fischermann, Rautenbach, and Volkmann and on the structures of planar graphs by Borodin and Ivanova. In particular, we identify additional structures and establish bounds on the independent bondage number for planar graphs with $\delta (G) \geq 2$ and $g(G)\geq 5$, $\delta(G)\geq 3$ and $g(G)\geq 4$, and $\delta (G) \geq 2$ and $g(G)\geq 10$.
\end{abstract}

\section{Introduction}
\label{introductiondefs}

In this paper, we exclusively consider finite, undirected, and simple planar graphs. For a given graph $G$, let $V(G)$ be the set of vertices of $G$, $E(G)$ be the set of edges, and $F(G)$ be the set of faces. For $u\in V(G)$, let $N_G(u)=N(u)$ denote the neighborhood of $u$, $N_G(u)=N[u]$ denote the closed neighborhood of $u$, and $N_G(X)=N(X)=\bigcup_{x \in X}N(x)$ for a set $X\subseteq V(G)$. The degree of a vertex $u$ is denoted by $d_G(u)=d(u)=|N(u)|$. The number of vertices incident with a face $f$ is defined as the face degree and denoted by $\ell(f)$. We use $\delta(G)$ to represent the minimum degree.

The girth of $G$ is defined as the length of the shortest cycle in $G$ and denoted by $g(G)$. If $G$ has no cycles, then the girth of the graph is considered to be infinite. A $j$-, $j^+$-, or $j^-$-vertex is defined as a vertex with degree exactly $j$, at least $j$, or at most $j$, respectively. A $j$-neighbor of a vertex $v$ is a $j$-neighbor adjacent to $v$. An $(a,b^-)$-edge is an edge $e=xy$ such that $d(x)=a$ and $d(y)\leq b$. In other words, an $(a,b^-)$-edge has weight at most $a+b$. The set of vertices in $G$ with degree exactly $i$, at least $i$, and at most $i$ are denoted by $S_i$, $S_{i^+}$, and $S_{i^-}$, respectively.

Vertices of degree $i$ with exactly $j$ 2-neighbors are denoted by $S(i,j)$, vertices of degree $i$ with at least $j$ 2-neighbors are denoted by $S(i,j^+)$, and vertices of degree at least $i$ with at least $j$ 2-neighbors are denoted by $S(i^+,j^+)$.

A dominating set of $G$ is a set $D\subseteq V(G)$ such that $D\cup N(D)=V(G)$. The domination number of $G$ is denoted by $\gamma (G)$ and is defined as the cardinality of the smallest dominating set. The bondage number $b(G)$ of a non-empty graph $G$ is the cardinality of the smallest set of edges whose removal from $G$ creates a graph with a domination number at least $\gamma(G)+1$. An independent dominating set is a set of vertices in $G$ that is both independent and dominating. The minimum cardinality among all independent dominating sets is called the independent domination number $\gamma_i(G)$. The independent bondage number of a graph $G$ is denoted by $b_i(G)$ and is defined as the minimum cardinality among all edge sets $B\subseteq G$ such that $\gamma_i(G-B)>\gamma_i(G)$.

Domination number and bondage number have been widely studied in the past few decades, resulting in various types of domination numbers and bondage numbers. In 1979, Garey and Johnson proved that determining the domination number for general graphs is an NP-complete problem. In 1983, Bauer et al. \cite{bauer1983domination} introduced the measure of the efficiency of a dominating set for the first time, named domination line-stability. It is defined as the minimum number of edges to be removed from $G$ to increase $\gamma$. In 1990, the bondage number was formally introduced by Fink et al. \cite{fink1990bondage} as a parameter for measuring the vulnerability of the interconnection network under a link failure. Results on the bondage number have been studied in respect to the degree sum of adjacent vertices by Hartnell and Rall \cite{hartnell1994bounds}, and in respect to the maximum degree of planar graphs by Kang and Yuan \cite{kang2000bondage}. The results on bondage numbers up to 2010 are presented in the survey paper by Xu \cite{xu2013bondage}.

The independent bondage number is a relatively new concept with fewer results available. In 2018, Priddy, Wang, and Wei \cite{priddy2019independent} determined the independent bondage number for several graphs and presented an upper bound for planar graphs with respect to the maximum degree. In 2021, Pham and Wei \cite{phamindependent} determined a constant upper bound for planar graphs with $\delta(G)\geq 3$. In the same paper, they constructed a class of planar graphs with $b_i(G)=6$ with $\delta(G)\geq 3$. 

In 2003, Fischermann et al. \cite{fischermann2003remarks} constructed an upper bound for connected planar graphs using girth conditions, proving that $b(G)\leq 6$ when $g(G)\geq 4$, $b(G)\leq 5$ when $g(G)\geq 5$, $b(G)\leq 4$ when $g(G)\geq 6$, and $b(G)\leq 3$ when $g(G)\geq 8$. Following their result, it is natural to see if we can determine constant upper bounds for the independent bondage number of planar graphs under girth constraints. In particular, we prove constant upper bounds for planar graphs with $\delta (G) \geq 2$ and $g(G)\geq 5$, $\delta(G)\geq 3$ and $g(G)\geq 4$, and $\delta (G) \geq 2$ and $g(G)\geq 10$. To achieve this goal, we employ the discharging method, a method well-suited for planar graphs that assigns and reallocates charge  to identify and prove the existence of certain structures.  Readers can refer to the survey paper by Cranston and West \cite{cranston2017introduction} for a better understanding of the discharging method.

This paper is organized as follows: In Section \ref{section1}, we discuss  preliminary results. Section \ref{section2} examines graphs with $\delta(G) \geq 2$ and $g(G) \geq 5$. Section \ref{section3} focuses on graphs with $\delta(G) \geq 3$ and $g(G) \geq 4$. In Section \ref{section4}, we study graphs with $\delta(G) \geq 2$ and $g(G) \geq 7$. Section \ref{section5} explores graphs with $\delta(G) \geq 2$ and $g(G) \geq 10$. Section \ref{conclusionsection} summarizes our results and discusses future work.

\section{Preliminary Results}
\label{section1}
In 1994, Hartnell and Rall \cite{hartnell1994bounds} proved an upper bound for the bondage number with respect to the degree sum of adjacent vertices, and have extended it to pairs of vertices with distance at most 2. A similar result for the independent bondage number was proved by Priddy, Wang, and Wei \cite{priddy2019independent} in 2019.
\begin{theorem}\cite{priddy2019independent}  
\label{PriddyWei}
If $G$ is a non empty graph, then\\
$$b_i(G)\leq min\{d(u)+d(v)-|N(u)\cap N(v)|-1:uv \in E(G)\}.$$

\end{theorem}
In 2003 by Fischermann et at. \cite{fischermann2003remarks} constructed upper bound for connected planar graph using girth conditions.
\begin{theorem} \cite{fischermann2003remarks}
\label{Fischermannthm}
For any connected planar graph $G$,
\begin{center}
    
$b(G)\leq \begin{cases}
        6, & \text{if }  g(G)\geq 4\\
        5, & \text{if }  g(G)\geq 5\\
        4, & \text{if }  g(G)\geq 6\\
        3, & \text{if }  g(G)\geq 8\\
        \end{cases}$
\end{center}
\end{theorem}
Priddy, Wang, and Wei \cite{priddy2019independent} proved an upper bound for the independent bondage number for connected planar graphs using the maximum degree $\Delta(G)$.
\begin{theorem}\cite{priddy2019independent}
If $G$ is a connected planar graph, then 
\begin{center}
$b_i(G)\leq \Delta(G)+2.$
\end{center}
\end{theorem}
As an extension of their work, Pham and Wei \cite{phamindependent} proved the following theorem.
\begin{theorem}\cite{phamindependent}
Let $G$ be a planar graph with $\delta(G)\geq3$, then
\begin{center}
$b_i(G)\leq min\{9, \Delta(G)+2\}$
\end{center}
\end{theorem}

To find our upper bounds under girth constraints, we use the discharging method. This discharging method is well suited for planar graphs and assigns charges to faces and vertices and reassign the charges to avoid a particular structure. Euler's formula provides three natural ways to assign charges to planar graphs.
\begin{proposition} \cite{cranston2017introduction}
\label{dischargingequations}
Let $V(G)$ and $F(G)$ be the set of vertices and faces respectively of a planar graph $G$. Denote $\ell(f)$ to be the length of a face $f$. Then the following equalities holds for $G$.

\begin{align*}
   &\sum_{\substack{v\in V(G)}}(d(v)-6) &+ \quad  &\sum_{\substack{f\in F(G)}}(2l(f)-6) & =  \quad -12  & \quad \text{    vertex charging}\\
   &\sum_{\substack{v\in V(G)}}(2d(v)-6) &+  \quad &\sum_{\substack{f\in F(G)}}(l(f)-6) & = \quad -12  & \quad \text{    face charging}\\
   &\sum_{\substack{v\in V(G)}}(d(v)-4)  &+  \quad &\sum_{\substack{f\in F(G)}}(l(f)-4)  &= \quad -8  & \quad \text{    balanced charging}
   \end{align*}

\end{proposition}
Borodin et al., \cite{borodin2019all} proved the following result for planar graphs using the discharging method.
\begin{theorem}\cite{borodin2019all}
Every planar graph with minimum degree at least 3 has at least one of the following:
\begin{enumerate}[(a).]
    \item[(a).] a 3-face incident with a $(3,10)-, (4,7)-$, or $(5,6)$-edge;
    \item[(b).] a 4-face incident either with two 3-vertices and another $5^-$- vertex or with a 3-vertex, two 4-vertices, and a froth vertex of degree at most five;
    \item [(c).]a 5-face incident with four 3-vertices and a fifth vertex of degree at most 5, where all parameters are best possible.
\end{enumerate}
\end{theorem}
As an extension of this theorem, Pham and Wei \cite{phamindependent} proved the following theorem.
\begin{theorem}\cite{phamindependent}
Every planar graph with minimum degree at least 3 has at least one of the following:
\begin{enumerate}[(a).]
    \item[(a).] a 3-face incident with a $(4,7^-)-$, a $(3,8^-)$ edge, or a $(5,6^-)$ edge;
    \item[(b).] a $(3,9^-)$ edge shared by two 3-faces;
    \item[(c).] a 10-vertex incident to at most one $4^+$-face, and adjacent to four 3-vertices each lying on two common 3-faces and a fifth vertex of degree at most five;
    \item[(d).] a 4-face incident either with two 3-vertices and another $5^-$ vertex or with a 3-vertex, two 4-vertices, and a forth vertex of degree at most five;
    \item[(e).] a 5-face incident with four 3-vertices and a fifth vertex of degree at most 5, where all parameters are best possible.
\end{enumerate}
\end{theorem}

Motivated by Theorem \ref{Fischermannthm}, we investigate whether similar constant upper bounds for the independent bondage number could be established under girth conditions. 
\section{Graphs with $\delta(G)\geq 2$ and $g(G)\geq 5$}
\label{section2}




\begin{theorem}{}
\label{mainthm}
Every planar graph $G$ with $\delta(G)\geq 2$ and $g(G)\geq 5$ has at least one of the following configurations:
\begin{enumerate}[(a)]
\item \label{girth5confa} a $(2,4^-)$- or $(3,3^-)$-edge;
\item \label{girth5confb} a vertex $v\in S(5^+,[d(v)-1]^+)$ such that its remaining neighbor is either a $4^-$-vertex or a vertex $u\in S(5^+,1^+)$; [see Figure \ref{fig:configurationbcase1} and Figure \ref{fig:configurationbcase2)} on Page $13$]
\item \label{girth5confc} a $5$-vertex sharing a common $5$-face with two $2$-neighbors; [see Figure \ref{configuration c} on Page $14$ ]
\item \label{girth5confd} a vertex $v\in S(6^+,[d(v)-1]^+)$ with a $5^-$-neighbor such that $v$ shares a common $5$-face with two $2$-neighbors;[see Figure \ref{configuration d} on Page $14$]
\item \label{girth5confe} a vertex $v\in S(5^+,[d(v)-2]^+)$ such that its remaining neighbors are both $3^-$-vertices;[see Figure \ref{configuration e} on Page $14$ ]
\item \label{girth5conff} a vertex $v\in S(5,3)$ with a $3$-neighbor $u$ and sharing a common $5$-face with a $2$-neighbor $v_1$, a $5^+$-neighbor $w$, and a $2$-vertex $w_1\in N(w)$;[see Figure \ref{configuration f} on Page $14$]
\item \label{girth5confg} a vertex $v\in S(5,3)$ sharing a common $5$-face with a $3$-neighbor $u$, a $5^+$-neighbor $w$, and a $2$-vertex $w_1\in N(w)$;[see Figure \ref{configuration g} on Page $15$]
\item \label{girth5confh} a vertex $v\in S(5,3)$ with two $5^+$-neighbors $u,w$ such that $u$ and $w$ each has a $2$-neighbor $u_1,w_1$ and $\{v,u,u_1,v_1\}$ and $\{v,w,w_1,v_2\}$ each lie on a $5$-face for $v_1,v_2\in N(v)\cap S_2$;[see Figure \ref{configuration h} on Page $15$]
\item \label{girth5confi} a vertex $v\in S(5,3)$ such that $v$ shares a common $5$-face with a $4$-neighbor $u$ and a $2$-neighbor $v_1$, and $v$ shares another $5$-face with a $5^+$-neighbor $w$, another $2$-neighbor $v_2$, and a $2$-vertex $w_1\in N(w)$;[see Figure \ref{configuration i} on Page $16$ ]
\item \label{girth5confj} a vertex $v\in S(5,3)$ with two $5^+$-neighbors $u,w$ such that $\{u,w\}$ does not share a common face, $w$ has at least one $2$-neighbor $w_1$, $u$ has at least two $2$-neighbors $u_1,u_2$, $\{v,u,u_1,v_1\}$ and $\{v,u,u_2,v_2\}$ each lie on a common $5$-face for $v_1,v_2\in N(v)\cap S_2$ and $v_1\neq v_2$, and the face incident to two $2$-neighbors of $v$ is a $6$-face containing three $2$-vertices.[see Figure \ref{configuration j} on Page $16$ ]
\end{enumerate}

\end{theorem}

\begin{proof}
Let $G$ be a non-trivial connected planar graph with $\delta(G)\geq 2$ and $g(G)\geq 5$. Suppose $G$ contains none of the listed configurations. We will proceed with balanced charging mentioned in Proposition \ref{dischargingequations} by assigning an initial charge of $d(v)-4$ to each vertex $v$ and $\ell(f)-4$ to each face $f$. We redistribute charge according to the following discharging rules:

\begin{enumerate}[(R1)]
\item Every $2$-vertex takes $\frac{1}{2}$ from each adjacent vertex and each incident face. 
\item Every $3$-vertex takes $\frac{1}{3}$ from each incident face. 
\item Every vertex in $S(6,5)$ takes $\frac{1}{12}$ from each incident $5$-face $f$ with $|V(f)\cap S_2|\leq 1$ and from each incident $6^+$-face, and $\frac{d(v)-4}{|N(v)\cap [S(6,5)\cup S(5,3^+)]|}$ from each adjacent $5^+$-vertex $v$ with $N(v)\cap S_2=\emptyset$. 
\item Every vertex in $S(5,3^+)$ takes $\frac{\ell(f)-4-\frac{|V(f)\cap S_2|}{2}-\frac{|V(f)\cap S_3|}{3}-\frac{|V(f)\cap S(6,5)|}{12}}{|V(f)\cap S(5,3^+)|}$ from each incident $5$-face with $|V(f)\cap S_2|\leq 1$ and from each incident $6^+$-face, and $\frac{d(v)-4}{|N(v)\cap [S(6,5)\cup S(5,3^+)]|}$ from each adjacent $5^+$-vertex $v$ with $N(v)\cap S_2=\emptyset$.
\end{enumerate}

We provide some facts and their respective proofs.

\begin{fact}
\label{girthfact1}
Every $4^-$-vertex ends with $0$ charge. 
\end{fact}

By rules $(R1)$ and $(R2)$, it is easy to see that every $3^-$-vertex ends with $0$ charge. As $G$ does not contain Configuration $(a)$, every $4$-vertex can only be adjacent to $3^+$-vertices. Thus, none of the rules apply to $4$-vertices. Hence, every $4$-vertex neither gains nor loses charge and ends with final charge $4-4=0$. \\

\begin{fact}
\label{girthfact2}
A $5^+$-vertex may possibly give charge to either a $2$-vertex or a $5^+$-vertex in $S(6,5)\cup S(5,3^+)$, but not both. In particular, when a $5^+$-vertex incident with at least one $2$-vertex, it gives charge to incident $2$-vertices and does not give charge to a $5^+$-vertex in $S(6,5)\cup S(5,3^+)$.
\end{fact}

By rules $(R3)$ and $(R4)$, a $5^+$-vertex $v$ gives charge to adjacent vertices in $S(6,5)\cup S(5,3^+)$ if and only if $N(v)\cap S_2 =\emptyset$.\\

\begin{fact}
\label{girthfact3}
Every $5$- or $6$-vertex $v$ with $|N(v)\cap S_2|\leq 2d(v)-8$ and every $7^+$-vertex end with a non-negative charge. 
\end{fact}

If $|N(v)\cap S_2|\geq 1$, then $v$ only gives charge to adjacent $2$-vertices by Fact \ref{girthfact2}. When $5\leq d(v)\leq 6$ with $|N(v)\cap S_2|\leq 2d(v)-8$, $v$ ends with final charge at least $(d(v)-4)-\frac{1}{2}(2d(v)-8)=0$. When $d(v)\geq 7$, then $v$ has at most $(d(v)-1)$ $2$-neighbors by Configuration $(b)$ and ends with final charge at least $(d(v)-4)-\frac{1}{2}(d(v)-1)=\frac{d(v)}{2}-\frac{7}{2}\geq 0$.

If $|N(v)\cap S_2|=0$, then $v$ may only give charge to adjacent vertices in $S(6,5)\cup S(5,3^+)$ by Fact \ref{girthfact2}. By choice of the distributed charge $\frac{d(v)-4}{|N(v)\cap [S(6,5)\cup S(5,3^+)]|}$ in rules $(R3)$ and $(R4)$, $v$ evenly splits its positive initial charge among all potential recipients. Hence, $v$ does not lose too much charge and ends with a non-negative final charge.\\

\begin{fact}
\label{girthfact4}
Every face $f$ contains at most $\left\lfloor\frac{\ell(f)}{2}\right\rfloor$ incident $3^-$-vertices. 
\end{fact}

As $G$ does not contain Configuration $(a)$, every $2$-vertex is adjacent only to $5^+$-vertices and every $3$-vertex is adjacent to only $4^+$-vertices. Hence, every face $f$ may only contain at most $\left\lfloor\frac{\ell(f)}{2}\right\rfloor$ incident $3^-$-vertices.\\

\begin{fact}
\label{girthfact5}
Let $f$ be a face of $G$. Then $$\ell(f)-4-\frac{|V(f)\cap S_2|}{2}-\frac{|V(f)\cap S_3|}{3}-\frac{|V(f)\cap S(6,5)|}{12}\geq 0.$$
\end{fact}

If $f$ is a $6^+$-face, it suffices to consider only our most extreme case where $|V(f)\cap S_2|+|V(f)\cap S_3|=\left\lfloor\frac{\ell(f)}{2}\right\rfloor$. It follows that $|V(f)\cap S(6,5)|\leq \left\lceil\frac{\ell(f)}{2}\right\rceil$. Hence,
\begin{align*}
\ell(f)-4-\frac{|V(f)\cap S_2|}{2}&-\frac{|V(f)\cap S_3|}{3}-\frac{|V(f)\cap S(6,5)|}{12}\\
&\geq \ell(f)-4-\frac{1}{2}\left(|V(f)\cap S_2|+|V(f)\cap S_3|\right)-\frac{|V(f)\cap S(6,5)|}{12}\\
&\geq \ell(f)-4-\frac{1}{2}\left\lfloor\frac{\ell(f)}{2}\right\rfloor-\frac{1}{12}\left\lceil\frac{\ell(f)}{2}\right\rceil\\
&\geq \frac{1}{4}
\end{align*}

when $\ell(f)\geq 6$.

Suppose $\ell(f)=5$. If $|V(f)\cap S_2|=2$, then $f$ does not contain any $3$ vertices by Configuration $(a)$. Furthermore, by Rules $(R3)$ and $(R4)$, $f$ only gives charge to its incident $2$-vertices. Hence, $$\ell(f)-4-\frac{|V(f)\cap S_2|}{2}-\frac{|V(f)\cap S_3|}{3}-\frac{|V(f)\cap S(6,5)|}{12}=(5-4)-2\left(\frac{1}{2}\right)=0.$$

If $|V(f)\cap S_2|=1$, then $|V(f)\cap S_3|\leq 1$ by Fact \ref{girthfact4}. By Configuration $(b)$, $|V(f)\cap S(6,5)|\leq 2$. In particular, either $|V(f)\cap S(6,5)|= 2$ with $|V(f)\cap S_3|=0$ or $|V(f)\cap S(6,5)|\leq 1$ with $|V(f)\cap S_3|\leq 1$. If $|V(f)\cap S(6,5)|= 2$ with $|V(f)\cap S_3|=0$, then 
$$\ell(f)-4-\frac{|V(f)\cap S_2|}{2}-\frac{|V(f)\cap S_3|}{3}-\frac{|V(f)\cap S(6,5)|}{12}= (5-4)-\frac{1}{2}-2\left(\frac{1}{12}\right)=\frac{1}{3}>0.$$ If $|V(f)\cap S(6,5)|\leq 1$ with $|V(f)\cap S_3|\leq 1$, then 
$$\ell(f)-4-\frac{|V(f)\cap S_2|}{2}-\frac{|V(f)\cap S_3|}{3}-\frac{|V(f)\cap S(6,5)|}{12}\geq (5-4)-\frac{1}{2}-\frac{1}{3}-\frac{1}{12}=\frac{1}{12}>0.$$

If $|V(f)\cap S_2|=0$, then $|V(f)\cap S(6,5)|=0$. Hence, $|V(f)\cap S_3|\leq 2$ by Fact \ref{girthfact4}. It follows that 
$$\ell(f)-4-\frac{|V(f)\cap S_2|}{2}-\frac{|V(f)\cap S_3|}{3}-\frac{|V(f)\cap S(6,5)|}{12}\geq(5-4)-\frac{2}{3}=\frac{1}{3}>0.$$
Thus, the inequality holds for all faces $f$. \\

\begin{fact}
\label{girthfact6}
Each face ends with non-negative charge. In particular, if $\ell(f)=5$ and $|V(f)\cap S_2|\leq 1$, then $f$ may distribute at least $\frac{1}{12}$ to each incident vertex $v\in V(f)\cap S(5,3^+)$. If $\ell(f)= 6$, then $f$ may distribute at least $\frac{1}{6}$ to each incident vertex $v\in V(f)\cap S(5,3^+)$. If $\ell(f)\geq 7$, then $f$ may distribute at least $\frac{3}{8}$ to each incident $v\in V(f)\cap S(5,3^+)$.
\end{fact}

By Fact \ref{girthfact5}, $\ell(f)-4-\frac{|V(f)\cap S_2|}{2}-\frac{|V(f)\cap S_3|}{3}-\frac{|V(f)\cap S(6,5)|}{12}\geq 0$ for any $6^+$-face and any $5$-face $f$ with $|V(f)\cap S_2|\leq 1$. Hence, $f$ does not lose too much charge from its incident $3^-$-vertices after applying rules $(R1)$ and $(R2)$. Applying rules $(R3)$ and $(R4)$, $f$ evenly distributes its remaining charge to all its incident vertices in $S(5,3^+)$. Thus, $f$ does not lose too much charge and ends with a non-negative final charge. 

Suppose $\ell(f)=5$ with $|V(f)\cap S_2|\leq 1$. We may assume that $|V(f)\cap S(5,3^+)|\geq 1$. By Fact \ref{girthfact4}, $|V(f)\cap S_2|+|V(f)\cap S_3|\leq 2.$ If $|V(f)\cap S_2|+|V(f)\cap S_3|=2$, then $|V(f)\cap S_3|=1$ or $|V(f)\cap S_3|=2.$ Consider that case that $|V(f)\cap S_3|=1.$ In this case either $|V(f)\cap S(5,3^+)|=2$ with $|V(f)\cap S(6,5)|=0$ or $|V(f)\cap S(5,3)|=1$ with $|V(f)\cap S(6,5)|=1.$ Otherwise, $G$ must contain configurations $(b)$ or $(g)$. When $|V(f)\cap S(5,3^+)|=2$ with $|V(f)\cap S(6,5)|=0$, $$\frac{\ell(f)-4-\frac{|V(f)\cap S_2|}{2}-\frac{|V(f)\cap S_3|}{3}-\frac{|V(f)\cap S(6,5)|}{12}}{|V(f)\cap S(5,3^+)|}\geq\frac{(5-4)-\frac{1}{2}-\frac{1}{3}}{2}=\frac{1}{12}.$$ When $|V(f)\cap S(5,3)|=1$ with $|V(f)\cap S(6,5)|=1$, $$\frac{\ell(f)-4-\frac{|V(f)\cap S_2|}{2}-\frac{|V(f)\cap S_3|}{3}-\frac{|V(f)\cap S(6,5)|}{12}}{|V(f)\cap S(5,3^+)|}\geq\frac{(5-4)-\frac{1}{2}-\frac{1}{3}-\frac{1}{12}}{1}=\frac{1}{12}.$$ For the case when $|V(f)\cap S_3|=2$, $|V(f)\cap S_2|=0.$ Note that, we avoid configurations $(b)$ and $(e)$ in $G$, $|V(f)\cap S(5,3^+)|\leq 2$ and $|V(f)\cap S(6,5)|=0$. Thus, $$\frac{\ell(f)-4-\frac{|V(f)\cap S_2|}{2}-\frac{|V(f)\cap S_3|}{3}-\frac{|V(f)\cap S(6,5)|}{12}}{|V(f)\cap S(5,3^+)|}\geq\frac{(5-4)-2*\frac{1}{3}}{1}=\frac{1}{6}.$$
If $|V(f)\cap S_2|+|V(f)\cap S_3|\leq 1$, then $0\leq |V(f)\cap S(6,5)|\leq 3$ as we assumed there exists at least one incident vertex in $S(5,3^+)$. Thus,
$$\frac{\ell(f)-4-\frac{|V(f)\cap S_2|}{2}-\frac{|V(f)\cap S_3|}{3}-\frac{|V(f)\cap S(6,5)|}{12}}{|V(f)\cap S(5,3^+)|}\geq\frac{(5-4)-\frac{1}{2}-\frac{|V(f)\cap S(6,5)|}{12}}{4-|V(f)\cap S(6,5)|}\geq\frac{1}{12}.$$

Suppose $\ell(f)\geq 6$. By Fact \ref{girthfact4}, $|V(f)\cap S_2|+|V(f)\cap S_3|\leq \left\lfloor\frac{\ell(f)}{2}\right\rfloor$. It suffices to consider the case for which $|V(f)\cap S_2|+|V(f)\cap S_3|=\left\lfloor\frac{\ell(f)}{2}\right\rfloor$.  It follows that $|V(f)\cap S(5,3^+)|+|V(f)\cap S(6,5)|\leq\left\lceil\frac{\ell(f)}{2}\right\rceil$. Consider the case where $|V(f)\cap S(5,3^+)|=\left\lceil\frac{\ell(f)}{2}\right\rceil$ and $|V(f)\cap S(6,5)|=0$. Then, 
$$\frac{\ell(f)-4-\frac{|V(f)\cap S_2|}{2}-\frac{|V(f)\cap S_3|}{3}-\frac{|V(f)\cap S(6,5)|}{12}}{|V(f)\cap S(5,3^+)|}\geq\frac{(\ell(f)-4)-\frac{1}{2}\left\lfloor\frac{\ell(f)}{2}\right\rfloor-0}{\left\lceil\frac{\ell(f)}{2}\right\rceil}.$$
In particular, when $\ell(f)=6$, $$\frac{\ell(f)-4-\frac{|V(f)\cap S_2|}{2}-\frac{|V(f)\cap S_3|}{3}-\frac{|V(f)\cap S(6,5)|}{12}}{|V(f)\cap S(5,3^+)|}\geq\frac{(\ell(f)-4)-\frac{1}{2}\left\lfloor\frac{\ell(f)}{2}\right\rfloor-0}{\left\lceil\frac{\ell(f)}{2}\right\rceil}=\frac{1}{6}$$
When $\ell(f)\geq 7$, $$\frac{\ell(f)-4-\frac{|V(f)\cap S_2|}{2}-\frac{|V(f)\cap S_3|}{3}-\frac{|V(f)\cap S(6,5)|}{12}}{|V(f)\cap S(5,3^+)|}\geq\frac{(\ell(f)-4)-\frac{1}{2}\left\lfloor\frac{\ell(f)}{2}\right\rfloor-0}{\left\lceil\frac{\ell(f)}{2}\right\rceil}\geq \frac{3}{8}.$$
Since $|V(f)\cap S(5,3^+)|+|V(f)\cap S(6,5)|\leq\left\lceil\frac{\ell(f)}{2}\right\rceil$ and vertices in $S(6,5)$ only take a charge of $\frac{1}{12}$ from $f$, then the replacement of an $S(5,3^+)$ vertex by an $S(6,5)$ vertex allows $f$ to give a charge greater than $\frac{3}{8}$ to its adjacent vertices in $S(5,3^+)$.\\ 
\begin{fact}
\label{girthfact7}
Given any $5^+$-vertex $v$ with $N(v)\cap S_2=\emptyset$, $\frac{d(v)-4}{|N(v)\cap [S(6,5)\cup S(5,3^+)]|}\geq \frac{1}{5}$.
\end{fact}

Since $v$ has at most $d(v)$ neighbors in $S(5,3^+)\cup S(6,5)$ and the function $f(x)=\frac{x-4}{x}$ is a strictly increasing function for all $x>0$, then $$\frac{d(v)-4}{|N(v)\cap [S(6,5)\cup S(5,3^+)]|}\geq\frac{d(v)-4}{d(v)}\geq  \frac{1}{5}$$ when $d(v)\geq 5$. \\

\begin{fact}
\label{girthfact8}
Every vertex $v\in S(6,5)$ ends with a non-negative final charge.
\end{fact}

The vertex $v$ only loses charge from adjacent $2$-vertices by Fact \ref{girthfact2}. After applying rule $(R1)$, $v$ has a charge of $(6-4)-5(\frac{1}{2})=-\frac{1}{2}$. It suffices to show that $v$ gains at least charge $\frac{1}{2}$ from its incident faces and from its neighbor $u\in N(v)\backslash S_2$. 

By Configuration $(b)$, $d(u)\geq 5$ and $N(u)\cap S_2=\emptyset$. If $d(u)=5$, then $v$ must be incident to at least four $6^+$-faces by Configuration $(d)$. By Fact \ref{girthfact7} and rule $(R3)$, the four incident $6^+$-faces and $u$ may contribute a total of at least charge $4\left(\frac{1}{12}\right)+\frac{1}{5}=\frac{8}{15}>\frac{1}{2}$ to $v$. 
Suppose $d(u)\geq 6$ and let $f_1$ and $f_2$ denote the faces incident to $vu$. Note that $N(u)\cap S_2 = \emptyset$ as we avoid the Configuration $(b)$. Thus, by rule $(R3)$, $u$ gives at least charge $\frac{(6-4)}{6} = \frac{1}{3}$ to $v$. If $f_i$ is a $5$-face for some $i=1,2$, then $|V(f_i)\cap S_2|=1$ by Configuration $(b)$. Hence, whether $f_i$ is a $5$- or $6^+$-face, $f_i$ is able to contribute at least charge $\frac{1}{12}$ to $v$ by rule $(R3)$. Thus $u$, $f_1$, and $f_2$ may give at least charge $\frac{1}{3}+2\left(\frac{1}{12}\right)=\frac{1}{2}$ to $v$. \\

Using Facts 1-8, we continue the main proof. As $G$ does not contain Configuration $(b)$, then $S(5,5)\cup S(6,6)=\emptyset$. By Facts \ref{girthfact1}, \ref{girthfact3}, \ref{girthfact6}, and \ref{girthfact8}, we only need to consider vertices in $S(5,3)\cup S(5,4)$. \\

Let $v\in S(5,4)$ and $u\in N(v)\backslash S_2$. By Configuration $(b)$, $u\in S(5^+,0)$. By Configuration $(c)$, $v$ must be incident to three $6^+$-faces that do not contain $u$. Each such $6^+$-face may give at least charge $\frac{1}{6}$ to $v$ by fact \ref{girthfact6}. After taking charge from the three incident $6^+$-faces, $v$ has at least charge $(5-4)-4\left(\frac{1}{2}\right)+3\left(\frac{1}{6}\right)=-\frac{1}{2}$. Let $f_1$ and $f_2$ denote the faces incident to $vu$. It suffices to show that $v$ gains at least charge $\frac{1}{2}$ from $f_1$, $f_2$, and $u$. We consider the follwoing three cases:

\textbf{Case 1: } Both $f_1$ and $f_2$ are $6^+$-faces.

As $u\in S(5^+,0)$, $u$ is able to give at least charge $\frac{1}{5}$ to $v$ by Fact \ref{girthfact7}. By Fact \ref{girthfact6}, $f_1$ and $f_2$ may give at least charge $\frac{1}{6}$ to $v$. Thus $v$ gains a charge of at least charge $\frac{1}{5}+2\left(\frac{1}{6}\right)=\frac{8}{15}>\frac{1}{2}$.  \\

\textbf{Case 2: } Either $f_1$ or $f_2$ is a $6^+$-face, but not both.

Without loss of generality, suppose $f_1$ is a $6^+$-face and $f_2$ is a $5$-face. Then $f_1$ may give at least charge $\frac{1}{6}$ to $v$. As $u\in S(5^+,0)$, then $f_2$ contains one $2$-vertex and at most one $3$-vertex. If $|V(f_2)\cap S_3|=0$, then $f_2$ may give at least charge $\frac{1}{6}$ by rule $(R4)$ and $u$ may give at least charge $\frac{1}{5}$ to $v$ by Fact \ref{girthfact7}. If $|V(f_2)\cap S_3|=1$, then $f_2$ may give at least charge $\frac{1}{12}$ and $u$ may give at least charge $\frac{1}{4}$ to $v$ by rule $(R4)$. In both subcases, $v$ receives at least charge $\frac{1}{2}$ from $f_1$, $f_2$, and $u$.\\

\textbf{Case 3: } Both $f_1$ and $f_2$ are $5$-faces.

As $u\in S(5^+,0)$, then $f_1$ and $f_2$ each contain one $2$-vertex and at most one $3$-vertex. If $|V(f_1)\cap S_3|=|V(f_2)\cap S_3|=0$, then $f_1$ and $f_2$ may each give at least charge $\frac{1}{6}$ by rule $(R4)$ as $|V(f_1) \cap S(5,3^+)|\leq 3$ and $|V(f_2)\cap S(5,3^+)|\leq 3$, and $u$ may give at least charge $\frac{1}{5}$ to $v$ by Fact \ref{girthfact7}. If $|V(f_1)\cap S_3|=|V(f_2)\cap S_3|=1$, then $f_1$ and $f_2$ may each give at least charge $\frac{1}{12}$ by rule $(R4)$ as $|V(f_1) \cap S(5,3^+)|\leq 2$ and $|V(f_2)\cap S(5,3^+)|\leq 2$, and $u$ may give at least charge $\frac{1}{3}$ to $v$ by rule $(R4)$ as $|N(u)\cap[S(6,5)\cup S(5,3^+]|\leq d(u)-2$ and $\frac{(d(u)-4)}{(d(u)-2)}\geq \frac{1}{3}$. If $|V(f_1)\cap S_3|=0$ and $|V(f_2)\cap S_3|=1$, then $f_1$ may give at least charge $\frac{1}{6}$, $f_2$ may give at least charge $\frac{1}{12}$, and $u$ may give at least charge $\frac{1}{4}$ to $v$ by rule $(R4)$. In all sub-cases, $v$ receives at least charge $\frac{1}{2}$ from $f_1$, $f_2$, and $u$.\\

Thus, every vertex $v\in S(5,4)$ ends with a nonnegative charge. For $v\in S(5,3)$, we show the following fact.

\begin{fact}
\label{girthfact9}
If $v\in S(5,3)$ is on at least three incident $6^+$-faces, then $v$ ends with a non-negative final charge. 
\end{fact}

For each $v\in S(5,3)$, $v$ only loses charge to adjacent $2$-vertices by Fact \ref{girthfact2}. After applying rule $(R1)$,  $v$ has a charge of $(5-4)-\frac{1}{2}(3)=-\frac{1}{2}$.  If $v$ is incident to at least three $6^+$-faces, then $v$ receives at least charge $(3)\frac{1}{6}=\frac{1}{2}$ from its incident $6^+$-faces by Fact \ref{girthfact6}. Hence, $v$ ends with a non-negative final charge. \\

Using our new Fact \ref{girthfact9}, we proceed with the main proof. Let $v\in S(5,3)$ and $u_1,u_2\in N(v)\backslash S_2$. We will consider two cases.

\textbf{Case 1: } $u_1$ and $u_2$ share a common face incident to $v$.

By Configuration $(c)$, the two incident faces each containing $2$-neighbors of $v$ must be $6^+$-faces. If $v$ has a third incident $6^+$-face, then we are done by Fact \ref{girthfact9}. Assume that $v$ has three incident $5$-faces. After receiving charge from the two incident $6^+$-faces, $v$ has at least charge $(5-4)-3\left(\frac{1}{2}\right)+2\left(\frac{1}{6}\right)=-\frac{1}{6}$. Thus, it suffices to show that $v$ gains at least charge $\frac{1}{6}$ from its remaining incident faces and neighbors.

Let $f_1$, $f_2$, and $f_3$ denote the $5$-faces incident to $v$ such that $u_i\in V(f_{i})\cap V(f_{i+1})$ for $1\leq i \leq 2$. If $d(u_i)\leq 4$ for $1\leq i \leq 2$, then $f_1$ and $f_3$ each has at most one incident $2$-vertex by Configuration $(a)$. By Fact \ref{girthfact6}, $f_1$ and $f_3$ may contribute a total of at least charge $2\left(\frac{1}{12}\right)=\frac{1}{6}$ to $v$. Hence, assume $d(u_1)\geq 5$. If $N(u_1)\cap S_2 = \emptyset$, then $u_1$ is able to contribute at least charge $\frac{1}{5}>\frac{1}{6}$ to $v$ by fact \ref{girthfact7}. Thus, assume $N(u_1)\cap S_2 \neq \emptyset$. By symmetry, we may assume $d(u_2)\geq 5$ and $N(u_2)\cap S_2 \neq \emptyset$ as well. It follows that $f_2$ contains at most one incident $2$-vertex and may give at least charge $\frac{1}{8}$ to $v$ by $(R4)$. By Configuration $(h)$, either $f_1$ and $f_3$ must contain exactly one incident $2$-vertex, and hence may give at least charge $\frac{1}{12}$ to $v$ by Fact \ref{girthfact6}. Therefore, $v$ receives at least charge $\frac{1}{8}+\frac{1}{12}=\frac{5}{24}>\frac{1}{6}$.\\

\textbf{Case 2: } $u_1$ and $u_2$ does not share a common face incident to $v$.

By Configuration $(c)$, the face containing two $2$-neighbors of $v$ must be a $6^+$-face. After receiving charge from its incident $6^+$-face, $v$ has at least charge $(5-4)-3\left(\frac{1}{2}\right)+\frac{1}{6}=-\frac{1}{3}$. Thus, it suffices to show that $v$ gains at least charge $\frac{1}{3}$ from its remaining incident faces and neighbors. If the face $f$ is incident to $v$ and two of its $2$-neighbors is a $7^+$-face, then $f$ may give at least charge $\frac{3}{8}>\frac{1}{3}$ to $v$ by Fact \ref{girthfact6}. Thus assume $f$ is a $6$-face. 

Let $f_1$, $f_2$, $f_3$, and $f_4$ denote the faces incident to $v$ such that $u_i\in V(f_{2i-1})\cap V(f_{2i})$ for $1\leq i \leq 2$ such that $f_2$ and $f_3$ are incident, and $x\in [N(v)\cap S_2]\cap [V(f_2)\cap V(f_3)]$. By Fact \ref{girthfact9}, if $v$ has two additional $6^+$-faces, then we are done. Thus, assume $v$ has at least three incident $5$-faces. By symmetry, we may assume $f_1$ and $f_3$ are $5$-faces. By Configuration $(e)$, $d(u_1)+d(u_2)\geq 7$. We consider the following subcases. 

\textbf{Case 2.1: } $d(u_1)\leq 4$ and $d(u_2)\leq 4$. 

By Configuration $(a)$, every $5$-face incident to $v$ must contain at most one incident $2$-vertex. Hence, $f_1$ and $f_3$ are able to each contribute at least charge $\frac{1}{12}$ to $v$ by Fact \ref{girthfact6}. If either $f_2$ or $f_4$ is a $6^+$-face, then $v$ receives at least charge $2\left(\frac{1}{12}\right)+\frac{1}{6}=\frac{1}{3}$ from its remaining incident faces. If both $f_2$ and $f_4$ are $5$-faces, then $v$ receives at least charge $4\left(\frac{1}{12}\right)=\frac{1}{3}$ from its incident $5$-faces by Fact \ref{girthfact6}.

\textbf{Case 2.2: } Either $d(u_1)\geq 5$ or $d(u_2)\geq 5$.

Without loss of generality, assume $d(u_1)\geq 5$. If $N(u_1)\cap S_2 =\emptyset$, then $u_1$ is able to contribute at least charge $\frac{1}{5}$ to $v$ by Fact \ref{girthfact7}. Furthermore, it follows that $f_1$ contains at most one incident $2$-vertex and can contribute at least charge $\frac{1}{12}$ to $v$ by Fact \ref{girthfact6}. Similarly, if $f_2$ is a $5$-face, then $f_2$ may also contribute at least charge $\frac{1}{12}$ to $v$. If $f_2$ is a $6^+$-face, then $f_2$ may contribute $\frac{1}{6}>\frac{1}{12}$ to $v$ by Fact \ref{girthfact6}. Thus, $v$ receives at least charge $\frac{1}{5}+2\left(\frac{1}{12}\right)=\frac{11}{30}>\frac{1}{3}$ from $f_1$, $f_2$, and $u_1$. Hence, assume $N(u_1)\cap S_2 \neq \emptyset$.

Suppose $3\leq d(u_2)\leq 4$. Then $|V(f_4)\cap S_2|= 1$ by Configuration $(a)$. Furthermore by Configurations $(f)$ and $(i)$, $|V(f_1)\cap S_2|=1$. Thus, $f_1$ and $f_3$ are each able to contribute charge $\frac{1}{12}$ to $v$. If either $f_2$ or $f_4$ is a $6^+$-face, then $v$ receives at least charge $2\left(\frac{1}{12}\right)+\frac{1}{6}=\frac{1}{3}$ from all $f_i$ for $1\leq i \leq 4$. Hence, assume $\ell(f_2)=\ell(f_4)=5$. By an analogous argument from $f_1$ and $f_3$ respectively, $f_2$ and $f_4$ are each able to contribute charge $\frac{1}{12}$ to $v$. Thus, $v$ receives at least charge $4\left(\frac{1}{12}\right)=\frac{1}{3}$ from all $f_i$ for $i\in \{1,2,3,4\}.$

Suppose $d(u_2)\geq 5$. By an analogous argument to that for $u_1$, we may assume $N(u_2)\cap S_2 \neq\emptyset$. If $v$ has no incident $5$-faces that contain two $2$-vertices, then $v$ receives at least charge $4\left(\frac{1}{12}\right)=\frac{1}{3}$ from all $f_i$. By Configuration $(h)$, $v$ may only have at most two incident $5$-faces that each contain two $2$-vertices. In particular, if  either $f_1$ or $f_2$ is a $5$-face with two incident $2$-vertices, then $f_3$ and $f_4$ each contain exactly one incident $2$-vertex. Similiarly, if  either $f_3$ or $f_4$ is a $5$-face with two incident $2$-vertices, then $f_1$ and $f_2$ each contain exactly one incident $2$-vertex. Consider two subcases:

\textbf{Case 2.2.1: } $|V(f_1)\cap S_2|=2$.

Then $|V(f_3)\cap S_2|=1$. Furthermore, there exists a vertex $y\in [V(f)\cap V(f_1)]\backslash[S_2\cup \{v\}]$, where $f$ is the $6$-face containing the two $2$-neighbors of $v$. Then $d(y)\geq 6$ by Configuration $(c)$. 

If $y\in S(6,5)$ and $|V(f)\cap S_2|=2$, then $f$ contains no $3$-vertices and at most two $S(5,3^+)$ vertices by Configuration $(b)$. Thus, $f$ may give at least charge $\frac{(6-4)-2\left(\frac{1}{2}\right)-\frac{1}{12}}{2}=\frac{11}{24}$ to $v$ (by rule $(R4)$) instead of its previous $\frac{1}{6}$. Hence, $v$ ends with final charge at least $(5-4)-3\left(\frac{1}{2}\right)+\frac{11}{24}+\frac{1}{12}=\frac{1}{24}>0$ after receiving charge from $f$ and $f_3$. 

If $y \in S(6,5)$ and $|V(f)\cap S_2|=3$, then $f$ may give at least charge $\frac{(6-4)-3\left(\frac{1}{2}\right)-\frac{1}{12}}{2}=\frac{5}{24}$ to $v$. If $f_2$ a $6^+$-face, then $f_4$ is a $5$-face with $|V(f_4)\cap S_2|=1$ by Configuration $(h)$. If $f_4$ a $6^+$-face, then $f_2$ is a $5$-face with $|V(f_4)\cap S_2|=1$ by Configuration $(j)$. Thus, in either case, $v$ ends with at least charge $(5-4)-3\left(\frac{1}{2}\right)+\frac{5}{24}+\frac{1}{6}+2\left(\frac{1}{12}\right)=\frac{1}{24}>0$ after receiving charge from $f$, $f_2$, $f_3$, and $f_4$. Hence assume, $f_2$ and $f_4$ are $5$-faces. As $|V(f_2)\cap S_2|=|V(f_3)\cap S_2|=1$, then $|V(f_2)\cap S(5^+, [d(v)-1]^+)|=|V(f_3)\cap S(5^+, [d(v)-1]^+)|=0$.

If $|V(f_2)\cap S_3|+|V(f_3)\cap S_3|=2$, then $|V(f_2)\cap S(5,3)|=0$ and $1\leq |V(f_3)\cap S(5,3)|\leq 2$ by configurations $(e)$ and $(f)$. Hence, $f_2$ may give $(5-4)-\frac{1}{2}-\frac{1}{3}=\frac{1}{6}$ to $v$. Thus $v$ ends with final charge at least $(5-4)-3\left(\frac{1}{2}\right)+\frac{5}{24}+\frac{1}{6}+2\left(\frac{1}{12}\right)=\frac{1}{24}$ after receiving charge from $f$, $f_2$, $f_3$, and $f_4$. Hence, assume $|V(f_2)\cap S_3|+|V(f_3)\cap S_3|\leq 1$. It follows that at least one of $f_2$ or $f_3$ contain at most four $S(5,3^+)$ vertices, and may give at least charge $\frac{(5-4)-\frac{1}{2}}{4}=\frac{1}8$ to $v$. Thus, $v$ ends with final charge at least $(5-4)-3\left(\frac{1}{2}\right)+\frac{5}{24}+\frac{1}{8}+2\left(\frac{1}{12}\right)=0$ after receiving charge from $f$, $f_2$, $f_3$, and $f_4$. 

Suppose $y\not \in S(6,5)$. If $|V(f)\cap S_2|=3$, then $f$ contains at most two $S(5,3^+)\cup S(6,5)$ vertices. Therefore, $f$ may give at least charge $\frac{(6-4)-3\left(\frac{1}{2}\right)}{2}=\frac{1}{4}$ to $v$ instead of its previous $\frac{1}{6}$. By Configuration $(j)$, either $f_2$ is a $6^+$-face or a $5$-face with exactly one incident $2$-vertex. Hence, $v$ ends with final charge at least $(5-4)-3\left(\frac{1}{2}\right)+\frac{1}{4}+3\left(\frac{1}{12}\right)=0$ after receiving charge from $f_2$, $f_3$, $f_4$, and $f$. If $|V(f)\cap S_2|=2$, then either $f$ contains one incident $3$-vertex and at most two incident $S(5,3^+)\cup S(6,5)$ vertices or no incident $3$-vertices and at most three in $S(5,3^+)\cup S(6,5)$ vertices.  In the first case, $v$ receives at least charge $\frac{6-4-2*\frac{1}{2}-\frac{1}{3}}{2}=\frac{1}{3}$, and in the second case $v$ receives at least charge $\frac{6-4-2*\frac{1}{2}}{3}=\frac{1}{3}$ from $f$. In both situations, $f$ may give at least charge $\frac{1}{3}$ to $v$.  Thus, $v$ ends with final charge at least $(5-4)-3\left(\frac{1}{2}\right)+\frac{1}{3}+2\left(\frac{1}{12}\right)=0$ after receiving charge from $f_3$, $f_4$, and $f$.

\textbf{Case 2.2.2: } $|V(f_3)\cap S_2|=2$.

Then $|V(f_1)\cap S_2|=1$. If $f_2$ is a $6^+$-face, then $f_4$ is a $5$-face. By symmetry and by Case 2.2.1, it is sufficient to consider when $|V(f_4)\cap S_2|=1$.  If $f_4$ is a $6^+$-face, then $f_2$ is a $5$-face with $|V(f_4)\cap S_2|=1$ by Configuration $(j)$. Therefore, in either situation, $v$ ends with final charge at least $(5-4)-3\left(\frac{1}{2}\right)+2\left(\frac{1}{12}\right)+2\left(\frac{1}{6}\right)=0$. Thus, assume $f_2$ and $f_4$ are both $5$-faces with $|V(f_2)\cap S_2|=|V(f_4)\cap S_2|=1$. Let $z\in [V(f_2)\cap V(f_3)]\backslash [S_2\cup \{v\}].$ By Configuration $(c)$, $d(y)\geq 6$. 

If $z\not\in S(6,5)$, then $f_2$ contains either one $3$-vertex and one $S(5,3)$-vertex by Configuration $(f)$ or no $3$-vertices and at most three $S(5,3)$-vertices. In either situations, $f_2$ may give at least charge $\frac{1}{6}$ to $v$. If $z\in S(6,5)$, then $f_2$ contains no $3$-vertices by Configuration $(b)$ and at most two in $S(5,3)$ vertices. Hence, $f_2$ may give at least charge $\frac{(5-4)-\frac{1}{2}-\frac{1}{12}}{2}=\frac{5}{24}>\frac{1}{6}$ to $v$. Therefore, regardless whether $z$ belongs to $S(6,5)$ or not, $f_2$ may give at least charge $\frac{1}{6}$ to $v$. It follows that $v$ ends with final charge at least $(5-4)-3\left(\frac{1}{2}\right)+2\left(\frac{1}{6}\right)+2\left(\frac{1}{12}\right)=0$ after taking charge from its incident faces. 

Thus, all vertices in  $S(5,3)$ end with non-negative charge. Therefore, all faces and vertices of $G$ end with a non-negative charge.
\end{proof}

Using the configurations found in Theorem 7, we determine a constant upper bound for the independent bondage number of planar graphs with $g(G)\geq 5$ and $\delta(G)\geq 2$.

\begin{theorem}
A planar graph $G$ with $g(G)\geq 5$ and $\delta(G)\geq 2$ has $b_i(G)\leq 5.$  
\label{girth5configurations}
\end{theorem}
\begin{proof}

Let $G$ a planar graph $G$ with $g(G)\geq 5$. By Theorem \ref{mainthm}, $G$ contains at least one of the configurations from $(a)-(j)$. If $G$ contains configuration $(a)$, then $b_i(G)\leq 5$ by Theorem \ref{PriddyWei}. Hence, we may assume $G$ contains at least one of the remaining configurations. For each configuration, we find an edge set $B\subset E(G)$ with $|B|\leq 5$ such that $|I|<|I'|$ where $I$ and $I'$ are minimum independent dominating sets of $G$ and $G'=G\backslash B$ respectively. Throughout the proof, we will suppose not; that is, $|I|\geq|I'|$.\\

\noindent \textbf{Configuration $(b)$.} Suppose $G$ contains configuration $(b)$. Let $v\in S(5^+,[d(v)-1]^+)$ and $v_i\in N(v)\cap S_2$ with $v_i'\in N(v_i)\backslash \{v\}$ for each $1\leq i \leq |N(v)\cap S_2|$. To avoid overlap with Configuration $(c)$, assume $v$ does not share a common $5$-face with any pair of its $2$-neighbors. 

We consider the first option of configuration $(b)$ and let $u\in N(v)\backslash S_2$ be a $4^-$-vertex (Figure \ref{fig:configurationbcase1}). Consider $B=\{ux|x\in N_G(u)\backslash \{v\}\}\cup \{v_1x | x\in N_G(v_1)\}$. Since $u$ is a $4^-$-vertex, then $|B|\leq 5$. Clearly, $v_1\in I'$. It follows that $N_G(v_1)\cap I'=\emptyset$. Otherwise, $I'\backslash \{v_1\}$ is an independent dominating set of $G$ with size $|I'|-1 < |I'|\leq |I|$, a contradiction as $I$ is a minimum independent dominating set of $G$. Hence, $v\notin I'$. Therefore, we must have $u\in I'$. Note that for each $2\leq i \leq |N(v)\cap S_2|$, either $|N_{G'}[v_i']\cap I'|=1$  or $|N_{G'}[v_i']\cap I'|\geq 2$.

Let $U$ be the set of all $v_i$ such that $|N_{G'}[v_i']\cap I'|=1$ and $W$ be the set of all $v_i$ such that $|N_{G'}[v_i']\cap I'|\geq 2$.  Then for each $v_i\in U$, $N_{G'}(v_i')\cap I'=\{v_i\}$, it follows that $[I'\backslash\{v_i\}]\cup\{v_i'\}$ is a minimum independent dominating set of $G'$. Thus take $U'=\{v_i'\}$, where $v_i \in U$. Note that $|U|=|U'|$. For each $v_i\in W$, there exists another vertex that dominates $v_i'$. Now, $I'\backslash (W\cup U \cup \{v_1, u\})\cup (\{v\}\cup U')$ is an independent dominating set of $G$ with size $|I'|-|U\cup W|-2+1 +|U|<|I'|\leq |I|$, a contradiction.



From our case of the first option of configuration $(b)$, we reference two claims that will often be used in our other cases. 

\textbf{Claim 1:} For each vertex $v\in I'$, $N_G(v)\cap I'=\emptyset$.\\

\textbf{Claim 2:} Given a vertex $v_i$, if $v_i'\in S_2\cap N(v)$, then we may assume either $N[v_i'] \cap I' = \{v_i'\}$ or $v_i\in I$ and there exists a vertex $y\in (N(v_i')\backslash \{v_i\})\cap I'$.\\

\begin{figure}[!htb]
   \begin{minipage}{0.48\textwidth}
     \centering
        \begin{tikzpicture}[scale=0.65]
   [
    > = stealth, 
    shorten > = 1pt, 
    auto,
    node distance = 2cm, 
    thick, dashed pattern=on 
    ]

    \tikzset{every state}=[
    draw = black,
    thick,
    fill = white,
    minimum size = 1mm
    ]
    \node[circle,draw=black, fill=black,scale=0.5, label = right :\small{v},minimum size=1pt](1) at (0,0){};
    \path (1) ++(18:2) node[circle,draw=black, fill=black,scale=0.5, label = below :\small{$v_1$}](2) {};
    \path (1) ++(90:2) node[circle,draw=black, fill=black,scale=0.5, label = right:\small{$u$}](3) {};
    \path (1) ++(162:2) node[circle,draw=black, fill=black,scale=0.5, label = below :\small{$v_4$}](4) {};
    \path (1) ++(234:2) node[circle,draw=black, fill=black,scale=0.5, label = below :\small{$v_3$}](5) {};
    \path (1) ++(306:2) node[circle,draw=black, fill=black,scale=0.5, label = below :\small{$v_2$}](6) {};
    
    \path (2) ++(30:2) node[circle,draw=black, fill=black,scale=0.5, label = below :\small{$v_1'$}](8) {};
    \path (3) ++(90:2) node[circle,draw=black, fill=black,scale=0.5](10) {};
    \path (4) ++(150:2) node[circle,draw=black, fill=black,scale=0.5](9) {};
    \path (5) ++(210:2) node[circle,draw=black, fill=black,scale=0.5](11) {};
    \path (6) ++(-30:2) node[circle,draw=black, fill=black,scale=0.5](12) {};
    \path (3) ++(25:2) node[circle,draw=black, fill=black,scale=0.5](14) {};
    \path (3) ++(155:2) node[circle,draw=black, fill=black,scale=0.5](15) {};


    \path [-,ultra thin] (1) edge node[left] {} (3);
    \path [-,ultra thin] (1) edge node[left] {} (4);
    \path [-,ultra thin] (1) edge node[left] {} (5);
    \path [-,ultra thin] (1) edge node[left] {} (6);

    \path [-,ultra thin] (4) edge node[left] {} (9);
    \path [-,ultra thin] (5) edge node[left] {} (11);
    
    \path [-,ultra thin] (6) edge node[left] {} (12);
    
    \path [-, thin,draw, dotted, color=black] (3) edge node[left] {} (15);
    \path [- , thin,draw, dotted, color=black] (3) edge node[left] {} (10);
    \path [-, thin,draw, dotted, color=black] (1) edge node[left] {} (2);
    \path [-, thin,draw, dotted, color=black] (2) edge node[left] {} (8);
    \path [-, thin,draw, dotted, color=black] (3) edge node[left] {} (14);



    \end{tikzpicture}

    \caption{Configuration $\ref{girth5confb}$, Option 1}
    \label{fig:configurationbcase1}
   \end{minipage}\hfill
   \begin{minipage}{0.48\textwidth}
     \centering
     \begin{tikzpicture}[scale=0.55]
   [
    > = stealth, 
    shorten > = 1pt, 
    auto,
    node distance = 2cm, 
    thick, dashed pattern=on 
    ]

    \tikzset{every state}=[
    draw = black,
    thick,
    fill = white,
    minimum size = 1mm
    ]
    \node[circle,draw=black, fill=black,scale=0.5, label = right :\small{v},minimum size=1pt](1) at (0,0){};
    \path (1) ++(18:2) node[circle,draw=black, fill=black,scale=0.5, label = below :\small{$v_1$}](2) {};
    \path (1) ++(90:2) node[circle,draw=black, fill=black,scale=0.5, label = right:\small{$u$}](3) {};
    \path (1) ++(162:2) node[circle,draw=black, fill=black,scale=0.5, label = below :\small{$v_4$}](4) {};
    \path (1) ++(234:2) node[circle,draw=black, fill=black,scale=0.5, label = below :\small{$v_3$}](5) {};
    \path (1) ++(306:2) node[circle,draw=black, fill=black,scale=0.5, label = below :\small{$v_2$}](6) {};
    
    \path (2) ++(30:2) node[circle,draw=black, fill=black,scale=0.5, label = below :\small{$v_1'$}](8) {};
    \path (3) ++(90:2) node[circle,draw=black, fill=black,scale=0.5](10) {};
    \path (4) ++(150:2) node[circle,draw=black, fill=black,scale=0.5](9) {};
    \path (5) ++(210:2) node[circle,draw=black, fill=black,scale=0.5](11) {};
    \path (6) ++(-30:2) node[circle,draw=black, fill=black,scale=0.5, label = below :\small{$v_2'$}](12) {};  
    \path (3) ++(25:2) node[circle,draw=black, fill=black,scale=0.5, label = below :\small{$u_1$}](13) {};
    \path (3) ++(155:2) node[circle,draw=black, fill=black,scale=0.5](14) {};
    \path (13) ++(50:2) node[circle,draw=black, fill=black,scale=0.5](15) {};


    \path [-,ultra thin] (1) edge node[left] {} (3);
    \path [-,ultra thin] (1) edge node[left] {} (4);
    \path [-,ultra thin] (1) edge node[left] {} (5);
    \path [-,ultra thin] (1) edge node[left] {} (6);

    \path [-,ultra thin] (4) edge node[left] {} (9);
    \path [-,ultra thin] (5) edge node[left] {} (11);
    
    \path [-, thin,draw, dotted, color=black] (6) edge node[left] {} (12);
    
    \path [-, thin,draw, dotted, color=black] (3) edge node[left] {} (13);
    \path [-,ultra thin]  (3) edge node[left] {} (10);
    \path [-, thin,draw, dotted, color=black] (1) edge node[left] {} (2);
    \path [-, thin,draw, dotted, color=black] (2) edge node[left] {} (8);
    \path [-,ultra thin] (3) edge node[left] {} (14);
    \path [-, thin,draw, dotted, color=black] (13) edge node[left] {} (15);



    \end{tikzpicture}
    \caption{Configuration $\ref{girth5confb}$, Option 2}
    \label{fig:configurationbcase2)}
   \end{minipage}
\end{figure}

We now consider the second option configuration $(b)$ and let $u\in N(v)\cap S(5^+,1^+)$. Then there exists $u_1\in (N(u)\backslash \{v\})\cap S_2$ (Figure \ref{fig:configurationbcase2)}). Let $B=\{u_1x|x\in N(u_1)\}\cup \{v_1x | x\in N(v_1)\}\cup \{v_2v_2'\}$. As $d(u_1)=d(v_1)=2$, then $|B|=5$. Clearly, $u_1,v_1\in I'$. 

By Claim 1, $N_G(v_1)\cap I'= \emptyset$. Thus, $v,u\notin I'$. It follows that $v_2\in I'$. By Claim 2, we assume either $N[v_i'] \cap I' = \{v_i'\}$ or there exists a vertex $y\in (N(v_i')\backslash \{v_i\})\cap I'$ with $v_i\in I'$ for each $3\leq i \leq |N_G(v)\cap S_2|$. Hence, $(I'\backslash N_G(v))\cup \{v\}$ is an independent dominating set of $G$ of size at most $|I'|-2+1<|I'|\leq |I|$, a contradiction.\\

\noindent \textbf{Configuration $(c)$.} Suppose $G$ contains configuration $(c)$. Let $v\in S_5$, $v_i \in N(v)\cap S_2$,  and $v_i'\in N(v_i)\backslash \{v\}$ for $i=1,2$ such that $\{v, v_1, v_2, v_1', v_2'\}$ lie on a common $5$-face (Figure \ref{configuration c}). Consider $B=\{vx | x\in N_G(v)\}$. Clearly, $v\in I'$. By Claim 1, $v_1,v_2\notin I'$. Hence, in order for $I'$ to be dominating, $v_1', v_2'\in I'$, a contradiction as $I'$ is independent.\\


\begin{figure}[!htb]
   \begin{minipage}{0.48\textwidth}
     \centering
     \begin{tikzpicture}[scale=0.65]
   [
    > = stealth, 
    shorten > = 1pt, 
    auto,
    node distance = 2cm, 
    thick, dashed pattern=on 
    ]

    \tikzset{every state}=[
    draw = black,
    thick,
    fill = white,
    minimum size = 1mm
    ]
    \node[circle,draw=black, fill=black,scale=0.5, label = right:\small{$v$}](1) at (0,0){};
    \path (1) ++(36:2) node[circle,draw=black, fill=black,scale=0.5, label = above:\small{$v_1$}](2) {};
    \path (1) ++(108:2) node[circle,draw=black, fill=black,scale=0.5, label = above:\small{$v_5$}](3) {};
    \path (1) ++(180:2) node[circle,draw=black, fill=black,scale=0.5, label = left:\small{$v_4$}](4) {};
    \path (1) ++(252:2) node[circle,draw=black, fill=black,scale=0.5, label = below:\small{$v_3$}](5) {};
    \path (1) ++(324:2) node[circle,draw=black, fill=black,scale=0.5, label = below:\small{$v_2$}](6) {};
    
    \path (2) ++(0:3) node[circle,draw=black, fill=black,scale=0.5, label =  right:\small{$v_1'$}](7) {};
    \path (6) ++(0:3) node[circle,draw=black, fill=black,scale=0.5, label =  right:\small{$v_2'$}](8) {};

    \path [-, thin,draw, dotted, color=black] (1) edge node[left] {} (2);
    \path [-, thin,draw, dotted, color=black] (1) edge node[left] {} (3);
    \path [-, thin,draw, dotted, color=black] (1) edge node[left] {} (4);
    \path [-, thin,draw, dotted, color=black] (1) edge node[left] {} (5);
    \path [-, thin,draw, dotted, color=black] (1) edge node[left] {} (6);
    
    \path [-,ultra thin] (2) edge node[left] {} (7);
    \path [-,ultra thin] (6) edge node[left] {} (8);
    \path [-,ultra thin] (7) edge node[left] {} (8);



    \end{tikzpicture}
    \caption{Configuration $\ref{girth5confc}$}
    \label{configuration c}
    \end{minipage}\hfill
    \begin{minipage}{0.48\textwidth}
     \centering
     \centering
    \begin{tikzpicture}[scale=0.65]
   [
    > = stealth, 
    shorten > = 1pt, 
    auto,
    node distance = 2cm, 
    thick, dashed pattern=on 
    ]

    \tikzset{every state}=[
    draw = black,
    thick,
    fill = white,
    minimum size = 1mm
    ]
    \node[circle,draw=black, fill=black,scale=0.5, label = below left:\small{$v$}](1) at (0,0){};
    \path (1) ++(30:2) node[circle,draw=black, fill=black,scale=0.5, label = below:\small{$v_1$}](2) {};
    \path (1) ++(80:2) node[circle,draw=black, fill=black,scale=0.5, label = above:\small{$v_3$}](3) {};
    \path (1) ++(130:2) node[circle,draw=black, fill=black,scale=0.5, label = above:\small{$v_4$}](15) {};
    
    \path (1) ++(180:3) node[circle,draw=black, fill=black,scale=0.5, label = below right:\small{$u$}](4) {};
    \path (1) ++(260:2) node[circle,draw=black, fill=black,scale=0.5, label = below:\small{$v_5$}](5) {};
    \path (1) ++(330:2) node[circle,draw=black, fill=black,scale=0.5, label = below:\small{$v_2$}](6) {};
    \path (4) ++(90:2) node[circle,draw=black, fill=black,scale=0.5](19) {};
    \path (4) ++(150:2) node[circle,draw=black, fill=black,scale=0.5](20) {};
    \path (4) ++(210:2) node[circle,draw=black, fill=black,scale=0.5](21) {};
    \path (4) ++(270:2) node[circle,draw=black, fill=black,scale=0.5](22) {};
    \path (2) ++(0:3) node[circle,draw=black, fill=black,scale=0.5, label = right:\small{$v_1'$}](7) {};
    \path (6) ++(0:3) node[circle,draw=black, fill=black,scale=0.5, label = right:\small{$v_2'$}](8) {};

    \path [-, thin,draw, dotted, color=black] (1) edge node[left] {} (4);
    \path [-, thin,draw, dotted, color=black] (4) edge node[left] {} (19);
    \path [-, thin,draw, dotted, color=black] (4) edge node[left] {} (20);
    \path [-, thin,draw, dotted, color=black] (4) edge node[left] {} (21);
    \path [-, thin,draw, dotted, color=black] (4) edge node[left] {} (22);

    \path [-,ultra thin] (1) edge node[left] {} (2);
    \path [-,ultra thin] (1) edge node[left] {} (3);
    \path [-,ultra thin] (1) edge node[left] {} (5);
    \path [-,ultra thin] (1) edge node[left] {} (6);
    \path [-,ultra thin] (1) edge node[left] {} (15);
    
    \path [-,ultra thin] (2) edge node[left] {} (7);
    \path [-,ultra thin] (6) edge node[left] {} (8);

    \path [-,ultra thin] (7) edge node[left] {} (8);



    \end{tikzpicture}

    \caption{Configuration $\ref{girth5confd}$}
    \label{configuration d}
   \end{minipage}
\end{figure}

\textbf{Configuration $(d)$.} Suppose $G$ contains configuration $(d)$. For $v\in S(6^+,d(v)-1)$, let $u\in N(v)$ be a $5^-$-neighbor and $v_i\in N(v)\cap S_2$ with $v_i'\in N(v_i)\backslash \{v\}$ for $1\leq i \leq d(v)-1$ such that $\{v_1,v_2, v\}$ lie on a common $5$-face (Figure \ref{configuration d}). Consider the edge set $B=\{ux | x\in N_G(u)\}$. Then $|B|\leq 5$. It is clear that $u\in I'$. By Claim 1, $v\not\in I'$. As $\{v,v_1,v_2,v_1', v_2'\}$ lie on a common $5$-face and $d(v_1)=d(v_2)=2$, it follows that at least one of \{$v_1,v_2\}$ belong to $I'$. Otherwise, $v_1',v_2'\in I'$, a contradiction as $I'$ is independent. Without loss of generality, we may assume $v_1\in I'$. By Claim 2, we assume either $N[v_i'] \cap I' = \{v_i'\}$ or there exists a vertex $y\in (N(v_i')\backslash \{v_i\})\cap I'$ with $v_i\in I'$ for each $2\leq i \leq d_G(v)-1$. Then $(I'\backslash (N_G(v)\cup \{v_1'\}))\cup (\{v\})$ is an independent dominating set of $G$ of size at most $|I'|-2+1<|I'|\leq |I|$, a contradiction. \\

\noindent \textbf{Configuration $(e)$.} Suppose $G$ contains configuration $(e)$. For $v\in S(5^+, (d(v)-2)^+)$, let $v_i\in N(v)\cap S_2$ with $v_i'\in N_G(v_i)\backslash \{v\}$ for each $1\leq i \leq d(v)-2$ and $u,w\in N(v)$ be the remaining $3^-$-vertices (Figure \ref{configuration e}). Consider the edge set $B=\{ux | x\in N(u)\}\cup \{wx | x\in N(w)\backslash \{v\}\}$. Then $3\leq |B| \leq 5$. Clearly, $u\in I'$.  By Claim 1, $v\not\in I'$. It follows that $w\in I'$. By Claim 2, we assume either $N[v_i'] \cap I' = \{v_i'\}$ or there exists a vertex $y\in (N(v_i')\backslash \{v_i\})\cap I'$ with $v_i\in I'$ for each $1\leq i \leq d_G(v)-2$. Then $(I'\backslash N_G(v))\cup \{v\}$ is an independent dominating set of $G$ of size at most $|I'|-2+1=|I'|-1<|I'|\leq |I|$, a contradiction. \\


\begin{figure}[!htb]
   \begin{minipage}{0.48\textwidth}
     \centering
    \begin{tikzpicture}[scale=0.55]
   [
    > = stealth, 
    shorten > = 1pt, 
    auto,
    node distance = 2cm, 
    thick, dashed pattern=on 
    ]

    \tikzset{every state}=[
    draw = black,
    thick,
    fill = white,
    minimum size = 1mm
    ]
    \node[circle,draw=black, fill=black,scale=0.5, label = below left:\small{$v$}](1) at (0,0){};
    \path (1) ++(36:2) node[circle,draw=black, fill=black,scale=0.5, label = below:\small{$u$}](2) {};
    \path (1) ++(108:2) node[circle,draw=black, fill=black,scale=0.5, label = left:\small{$v_1$}](3) {};
    \path (1) ++(180:2) node[circle,draw=black, fill=black,scale=0.5, label = below:\small{$v_2$}](4) {};
    \path (1) ++(252:2) node[circle,draw=black, fill=black,scale=0.5, label = right:\small{$v_3$}](5) {};
    
    \path (1) ++(324:2) node[circle,draw=black, fill=black,scale=0.5, label = left:\small{$w$}](6) {};
    
    \path (2) ++(0:2) node[circle,draw=black, fill=black,scale=0.5](7) {};
    \path (2) ++(72:2) node[circle,draw=black, fill=black,scale=0.5](8) {};
    \path (3) ++(108:2) node[circle,draw=black, fill=black,scale=0.5, label = left:\small{$v_1'$}](9) {};
    
    \path (6) ++(288:2) node[circle,draw=black, fill=black,scale=0.5](12) {};
    \path (6) ++(0:2) node[circle,draw=black, fill=black,scale=0.5](13) {};

    \path [-, thin,draw, dotted, color=black] (6) edge node[left] {} (12);
    \path [-, thin,draw, dotted, color=black] (1) edge node[left] {} (2);
    \path [-, thin,draw, dotted, color=black] (2) edge node[left] {} (7);
    \path [-, thin,draw, dotted, color=black] (2) edge node[left] {} (8);
    \path [-, thin,draw, dotted, color=black] (6) edge node[left] {} (13);
    
    \path [-,ultra thin] (3) edge node[left] {} (9);
    \path [-,ultra thin] (1) edge node[left] {} (4);
    \path [-,ultra thin] (1) edge node[left] {} (5);
    \path [-,ultra thin] (1) edge node[left] {} (6);
    \path [-,ultra thin] (1) edge node[left] {} (3);



    \end{tikzpicture}

    \caption{Configuration $\ref{girth5confe}$}
    \label{configuration e}
   \end{minipage}\hfill
   \begin{minipage}{0.48\textwidth}
     \centering
    \begin{tikzpicture}[scale=0.65]
   [
    > = stealth, 
    shorten > = 1pt, 
    auto,
    node distance = 2cm, 
    thick, dashed pattern=on 
    ]

    \tikzset{every state}=[
    draw = black,
    thick,
    fill = white,
    minimum size = 1mm
    ]
    \node[circle,draw=black, fill=black,scale=0.5, label = below:\small{$v$}](1) at (0,0){};
    \path (1) ++(36:2) node[circle,draw=black, fill=black,scale=0.5, label = below:\small{$w$}](2) {};
    \path (1) ++(108:2) node[circle,draw=black, fill=black,scale=0.5, label = right:\small{$u$}](3) {};
    \path (3) ++(78:2) node[circle,draw=black, fill=black,scale=0.5, label = right:\small{$u_1$}](12) {};
    \path (3) ++(138:2) node[circle,draw=black, fill=black,scale=0.5, label = left:\small{$u_2$}](13) {};
    \path (1) ++(180:2) node[circle,draw=black, fill=black,scale=0.5, label = below:\small{$v_3$}](4) {};
    
    \path (1) ++(252:2) node[circle,draw=black, fill=black,scale=0.5, label = right:\small{$v_2$}](5) {};
    
    \path (1) ++(324:2) node[circle,draw=black, fill=black,scale=0.5, label = above:\small{$v_1$}](6) {};
    
    \path (2) ++(0:3) node[circle,draw=black, fill=black,scale=0.5, label = right:\small{$w_1$}](7) {};
    \path (6) ++(0:3) node[circle,draw=black, fill=black,scale=0.5, label = right:\small{$v_1'$}](8) {};
    \path [-, thin,draw, dotted, color=black] (1) edge node[left] {} (3);

    \path [-,ultra thin] (1) edge node[left] {} (2);
    \path [-,ultra thin] (1) edge node[left] {} (6);
    \path [-,ultra thin] (1) edge node[left] {} (4);
    \path [-,ultra thin] (1) edge node[left] {} (5);
    \path [-,ultra thin] (6) edge node[left] {} (8);
    
    \path [-, thin,draw, dotted, color=black] (2) edge node[left] {} (7);
    \path [-, thin,draw, dotted, color=black] (3) edge node[left] {} (12);
    \path [-, thin,draw, dotted, color=black] (3) edge node[left] {} (13);
    \path [-, thin,draw, dotted, color=black] (8) edge node[left] {} (7);



    \end{tikzpicture}

    \caption{Configuration $\ref{girth5conff}$}
    \label{configuration f}
   \end{minipage}
\end{figure}

\noindent \textbf{Configuration $(f)$.} Suppose $G$ contains configuration $(f)$. Let $v\in S(5,3)$, $v_i\in N(v)\cap S_2$ with $v_i'\in N(v_i)\backslash \{v\}$ for $1\leq i \leq 3$, $u\in N(v)\cap S_3$ with $u_j\in N(u)\backslash \{v\}$ for $j=1,2$, and $w\in N(v)\cap S_{5^+}$ with $w_1\in N(w)\cap S_2$ such that $\{v, v_1, v_1', w, w_1\}$ lie on a common $5$-face (Figure \ref{configuration f}). Consider the edge set $B=\{ux | x\in N(u)\}\cup \{w_1x| x\in N(w_1)\}$. Then $|B|=5$. Clearly, $u,w_1\in I'$. By Claim 1, $v, w, v_1'\not\in I'$. Therefore, $v_1\in I'$. By Claim 2, we assume either $N[v_i'] \cap I' = \{v_i'\}$ or there exists a vertex $y\in (N(v_i')\backslash \{v_i\})\cap I'$ with $v_i\in I'$ for each $i=2,3$. Hence, $(I'\backslash N_G(v)) \cup \{v\}$ is an independent dominating set of $G$ with size at most $|I'|-2+1=|I'|-1<|I'|\leq |I|$, a contradiction.\\

 
\noindent \textbf{Configuration $(g)$.} Suppose $G$ contains configuration $(g)$. Let $v\in S(5,3)$, $v_i\in N(v)\cap S_2$ with $v_i'\in N(v_i)\backslash \{v\}$ for each $i=1,2,3$, $u\in N(v)\cap S_3$ with $u_j\in N(u)\backslash \{v\}$ for $j=1,2$, and $w\in N(v)\cap S_{5^+}$ with $w_1\in N(w)\cap S_2$ such that $\{v,u,u_1,w,w_1\}$ lie on a common $5$-face. (Figure \ref{configuration g}). Consider $B=\{v_1x | x\in N(v_1)\}\cup \{w_1x| x\in N(w_1)\}\cup \{uu_2\}$. Then $|B|=5$. Let $I$ and $I'$ be minimum independent dominating sets of $G$ and $G'=G-B$ respectively. Clearly, $v_1,w_1\in I'$. By Claim 1, $N_G(v_1)\cap I'=\emptyset$ and $N_G(w_1)\cap I'=\emptyset$. Hence, it must be that $u\in I'$.  By Claim 2, we assume either $N[v_i'] \cap I' = \{v_i'\}$ or there exists a vertex $y\in (N(v_i')\backslash \{v_i\})\cap I'$ with $v_i\in I'$ for each $i=1,2,3$. Then $(I'\backslash N_G(v))\cup \{v\}$ is an independent dominating set of $G$ of size at most $|I'|-2+1=|I'|-1<|I'|\leq |I|$, a contradiction. \\


\begin{figure}[!htb]
   \begin{minipage}{0.50\textwidth}
     \centering
    \begin{tikzpicture}[scale=0.65]
   [
    > = stealth, 
    shorten > = 1pt, 
    auto,
    node distance = 2cm, 
    thick, dashed pattern=on 
    ]

    \tikzset{every state}=[
    draw = black,
    thick,
    fill = white,
    minimum size = 1mm
    ]
    \node[circle,draw=black, fill=black,scale=0.5, label = below:\small{$v$}](1) at (0,0){};
    \path (1) ++(36:2) node[circle,draw=black, fill=black,scale=0.5, label = below:\small{$u$}](2) {};
    \path (1) ++(108:2) node[circle,draw=black, fill=black,scale=0.5, label = left:\small{$v_1$}](3) {};
    \path (3) ++(108:2) node[circle,draw=black, fill=black,scale=0.5, label = left:\small{$v_1'$}](12) {};
    \path (2) ++(90:2) node[circle,draw=black, fill=black,scale=0.5, label = left:\small{$u_2$}](13) {};
    \path (1) ++(180:2) node[circle,draw=black, fill=black,scale=0.5, label = below:\small{$v_2$}](4) {};
    \path (1) ++(252:2) node[circle,draw=black, fill=black,scale=0.5, label = right:\small{$v_3$}](5) {};
    
    \path (1) ++(324:2) node[circle,draw=black, fill=black,scale=0.5, label = above:\small{$w$}](6) {};
    
    \path (2) ++(0:3) node[circle,draw=black, fill=black,scale=0.5, label = below left:\small{$u_1$}](7) {};
    
    \path (6) ++(0:3) node[circle,draw=black, fill=black,scale=0.5, label = above left:\small{$w_1$}](8) {};
    \path [-, thin,draw, dotted, color=black] (1) edge node[left] {} (3);
    \path [-,ultra thin] (2) edge node[left] {} (7);
    \path [-,ultra thin] (1) edge node[left] {} (2);
    \path [-,ultra thin] (1) edge node[left] {} (6);
    \path [-,ultra thin] (1) edge node[left] {} (4);
    \path [-,ultra thin] (1) edge node[left] {} (5);
    
    \path [-, thin,draw, dotted, color=black] (3) edge node[left] {} (12);
    \path [-, thin,draw, dotted, color=black] (2) edge node[left] {} (13);
    \path [-, thin,draw, dotted, color=black] (6) edge node[left] {} (8);
    \path [-, thin,draw, dotted, color=black] (8) edge node[left] {} (7);



    \end{tikzpicture}

    \caption{Configuration $\ref{girth5confg}$}
    \label{configuration g}
   \end{minipage} \hfill
   \begin{minipage}{0.50\textwidth}
    \centering
    \begin{tikzpicture}[scale=0.65]
   [
    > = stealth, 
    shorten > = 1pt, 
    auto,
    node distance = 2cm, 
    thick, dashed pattern=on 
    ]

    \tikzset{every state}=[
    draw = black,
    thick,
    fill = white,
    minimum size = 1mm
    ]
    \node[circle,draw=black, fill=black,scale=0.5, label = below:\small{$v$}](1) at (0,0){};
    \path (1) ++(36:2) node[circle,draw=black, fill=black,scale=0.5, label = above:\small{$v_1$}](2) {};
    \path (1) ++(108:2) node[circle,draw=black, fill=black,scale=0.5, label = right:\small{$v_2$}](3) {};
    \path (1) ++(180:2) node[circle,draw=black, fill=black,scale=0.5, label = below right:\small{$w$}](4) {};
    \path (1) ++(252:2) node[circle,draw=black, fill=black,scale=0.5, label = right:\small{$v_3$}](5) {};
    \path (1) ++(324:2) node[circle,draw=black, fill=black,scale=0.5, label = above:\small{$u$}](6) {};

    \path (2) ++(36:2) node[circle,draw=black, fill=black,scale=0.5, label = above:\small{$v_1'$}](7) {};
    \path (3) ++(108:2) node[circle,draw=black, fill=black,scale=0.5, label = right:\small{$v_2'$}](8) {};
    
    \path (4) ++(108:2) node[circle,draw=black, fill=black,scale=0.5, label = left:\small{$w_1$}](10) {};

    \path (6) ++(36:2) node[circle,draw=black, fill=black,scale=0.5, label = below:\small{$u_1$}](15) {};

    \path [-,ultra thin] (1) edge node[left] {} (3);
    \path [-,ultra thin] (1) edge node[left] {} (4);
    \path [-,ultra thin] (1) edge node[left] {} (5);
    \path [-,ultra thin] (1) edge node[left] {} (6);
    \path [-,ultra thin] (2) edge node[left] {} (7);
    \path [-,ultra thin] (3) edge node[left] {} (8);

    \path [-, thin,draw, dotted, color=black] (10) edge node[left] {} (8);
    \path [-, thin,draw, dotted, color=black] (4) edge node[left] {} (10);
    \path [-, thin,draw, dotted, color=black] (15) edge node[left] {} (7);
    \path [-, thin,draw, dotted, color=black] (2) edge node[left] {} (1);
    \path [-, thin,draw, dotted, color=black] (15) edge node[left] {} (6);



    \end{tikzpicture}

    \caption{Configuration $\ref{girth5confh}$}
    \label{configuration h}
   \end{minipage}
\end{figure}

\noindent \textbf{Configuration $(h)$.} Suppose $G$ contains configuration $(h)$. Let $v\in S(5,3)$, $v_i\in N(v)\cap S_2$ with $v_i'\in N(v_i)\backslash \{v\}$ for $i=1,2,3$, and $u,w\in N(v)\cap S_{5^+}$ with $u_1\in N(u)\cap S_2$ and $w_1\in N(w)\cap S_2$ such that $\{v,u,u_1,v_1,v_1'\}$ and $\{v,w,w_1,v_2,v_2'\}$ each lie on common $5$-faces (Figure \ref{configuration h}). Let $B=\{u_1x|x\in N(u_1)\}\cup \{w_1x|x\in N(w_1)\}\cup \{vv_1\}$. Then $|B|=5$. Clearly, $w_1,u_1\in I'$. By Claim 1, $u,w,v_1',v_2'\notin I'$. As $v_1,v_2\in S_2$, then it must be that $v_1\in I'$. Thus, $v_2\in I'$ as $v\not\in I'$ by Claim 1. By Claim 2, we assume either $N[v_3'] \cap I' = \{v_3'\}$ or there exists a vertex $y\in (N(v_3')\backslash \{v_3\})\cap I'$ with $v_3\in I'$. Then $(I'\backslash N_G(v))\cup \{v\}$ is an independent dominating set of $G$ of size at most $|I'|-2+1=|I'|-1<|I'|\leq |I|$, a contradiction. \\
    

\noindent \textbf{Configuration $(i)$.} Suppose $G$ contains configuration $(i)$. Let $v\in S(5,3)$, $v_i\in N(v)\cap S_2$ with $v_i'\in N(v_i)\backslash \{v\}$ for $i=1,2,3$, $u\in N(v)\cap S_4$ with $u_j\in N(u)\backslash \{v\}$ for $j=1,2,3$, $w\in N(v)\cap S_{5^+}$ with $w_1\in N(w)\cap S_2$ such that $\{v,v_2,v_2',w,w_1\}$ and $\{v,v_1,v_1',u,u_1\}$ each lie on common $5$-faces (Figure \ref{configuration i}). Consider $B=\{v_2x|x\in N(v_2)\}\cup\{ww_1,uu_2,uu_3\}$. Then $|B|=5$. Clearly, $v_2\in I'$. By Claim 1, $N_G(v_2)\cap I'=\emptyset$. It follows that $w_1\in I'$ and $w\notin I'$. By Claim 2, we assume either $N[v_3'] \cap I' = \{v_3'\}$ or there exists a vertex $y\in (N(v_3')\backslash \{v_3\})\cap I'$ with $v_3\in I'$. 

We must have that at most one of $\{u_1,v_1'\}$ belongs to $I'$. Otherwise, we have a contradiction if both $u,v_1'\in I'$ as $I'$ is independent. If exactly one of $\{u_1,v_1'\}$ belongs to $I'$, then $(I'\backslash N_G(v))\cup \{v\}$ is an independent dominating set of $G$ with size at most $|I'|-1<|I'|\leq |I|$, a contradiction. Hence, assume both $u_1,v_1'\notin I'$. Thus, $v_1,u\in I'$. If $N_{G'}(u_1)\cap I'=\{u\}$, then $(I'\backslash N_G(v))\cup \{v,u_1\}$ is an independent dominating set of $G$ with size at most $|I'|-3+2<|I'|\leq |I|$, a contradiction. Then there exists some other vertex in $N(u_1)\cap (I'\backslash \{u\})$. By an analogous argument, there exists some vertex in $N(v_1')\cap (I'\backslash \{v_1\})$. Hence, $(I'\backslash N_G(v))\cup \{v\}$ is an independent dominating set of size at least $|I'|-3+1<|I'|\leq |I|$, a contradiction.\\

\begin{figure}[!htb]
   \begin{minipage}{0.48\textwidth}
     \centering
     \centering
    \begin{tikzpicture}[scale=0.65]
   [
    > = stealth, 
    shorten > = 1pt, 
    auto,
    node distance = 2cm, 
    thick, dashed pattern=on 
    ]

    \tikzset{every state}=[
    draw = black,
    thick,
    fill = white,
    minimum size = 1mm
    ]
    \node[circle,draw=black, fill=black,scale=0.5, label = below:\small{$v$}](1) at (0,0){};
    \path (1) ++(36:2) node[circle,draw=black, fill=black,scale=0.5, label = above:\small{$v_1$}](2) {};
    \path (1) ++(108:2) node[circle,draw=black, fill=black,scale=0.5, label = right:\small{$v_2$}](3) {};
    \path (1) ++(180:2) node[circle,draw=black, fill=black,scale=0.5, label = below right:\small{$w$}](4) {};
    \path (1) ++(252:2) node[circle,draw=black, fill=black,scale=0.5, label = left:\small{$v_3$}](5) {};
    \path (1) ++(324:2) node[circle,draw=black, fill=black,scale=0.5, label = below:\small{$u$}](6) {};
    \path (6) ++(324:2) node[circle,draw=black, fill=black,scale=0.5, label = below:\small{$u_2$}](16) {};

    \path (2) ++(36:2) node[circle,draw=black, fill=black,scale=0.5, label = above:\small{$v_1'$}](7) {};
    \path (3) ++(108:2) node[circle,draw=black, fill=black,scale=0.5, label = right:\small{$v_2'$}](8) {};
    
    \path (4) ++(108:2) node[circle,draw=black, fill=black,scale=0.5, label = left:\small{$w_1$}](10) {};
    \path (6) ++(252:2) node[circle,draw=black, fill=black,scale=0.5, label = below:\small{$u_3$}](14) {};
    \path (6) ++(36:2) node[circle,draw=black, fill=black,scale=0.5, label = right:\small{$u_1$}](15) {};
  
    \path [-,ultra thin] (1) edge node[left] {} (2);
    \path [-,ultra thin] (1) edge node[left] {} (4);
    \path [-,ultra thin] (1) edge node[left] {} (5);
    \path [-,ultra thin] (1) edge node[left] {} (6);

    \path [-,ultra thin] (2) edge node[left] {} (7);
    \path [-,ultra thin] (8) edge node[left] {} (10);
    \path [-,ultra thin] (6) edge node[left] {} (15);
    \path [-,ultra thin] (15) edge node[left] {} (7);
    \path [-, thin,draw, dotted, color=black] (4) edge node[left] {} (10);
    \path [-, thin,draw, dotted, color=black] (1) edge node[left] {} (3);
    \path [-, thin,draw, dotted, color=black] (3) edge node[left] {} (8);
    \path [-, thin,draw, dotted, color=black] (6) edge node[left] {} (16);
     \path [-, thin,draw, dotted, color=black] (6) edge node[left] {} (14);



    \end{tikzpicture}

    \caption{Configuration $\ref{girth5confi}$}
    \label{configuration i}
   \end{minipage}\hfill
   \begin{minipage}{0.48\textwidth}
      \centering
    \begin{tikzpicture}[scale=0.55]
   [
    > = stealth, 
    shorten > = 1pt, 
    auto,
    node distance = 2cm, 
    thick, dashed pattern=on 
    ]

    \tikzset{every state}=[
    draw = black,
    thick,
    fill = white,
    minimum size = 1mm
    ]
    \node[circle,draw=black, fill=black,scale=0.5, label = below :\small{$v$}](1) at (0,0){};
    \path (1) ++(36:2) node[circle,draw=black, fill=black,scale=0.5, label = below:\small{$v_1$}](2) {};
    \path (1) ++(108:2) node[circle,draw=black, fill=black,scale=0.5, label = right:\small{$v_3$}](3) {};
    \path (1) ++(180:2) node[circle,draw=black, fill=black,scale=0.5, label = below right:\small{$w$}](4) {};
    \path (1) ++(252:2) node[circle,draw=black, fill=black,scale=0.5, label = right:\small{$v_2$}](5) {};
    \path (1) ++(324:2) node[circle,draw=black, fill=black,scale=0.5, label = above:\small{$u$}](6) {};

    \path (2) ++(36:2) node[circle,draw=black, fill=black,scale=0.5, label = right:\small{$v_1'$}](7) {};
    \path (3) ++(108:2) node[circle,draw=black, fill=black,scale=0.5, label = left:\small{$v_3'$}](8) {};

    \path (4) ++(180:2) node[circle,draw=black, fill=black,scale=0.5, label = below:\small{$w_1$}](11) {};
    \path (11) ++(180:2) node[circle,draw=black, fill=black,scale=0.5](30) {};

    \path (5) ++(252:2) node[circle,draw=black, fill=black,scale=0.5, label =above left:\small{$v_2'$}](13) {};

    \path (6) ++(252:2) node[circle,draw=black, fill=black,scale=0.5, label = below:\small{$u_2$}](14) {};
    \path (6) ++(36:2) node[circle,draw=black, fill=black,scale=0.5, label = below:\small{$u_1$}](15) {};
    
    \path (1) ++(72:5) node[circle,draw=black, fill=black,scale=0.5, label = above:\small{$p$}](9) {};

    \path [-,ultra thin] (1) edge node[left] {} (3);
    \path [-,ultra thin] (1) edge node[left] {} (4);
    \path [-,ultra thin] (1) edge node[left] {} (5);
    \path [-,ultra thin] (1) edge node[left] {} (6);

    \path [-,ultra thin] (2) edge node[left] {} (7);
    \path [-,ultra thin] (3) edge node[left] {} (8);   
    \path [-,ultra thin] (8) edge node[left] {} (9);
    \path [-,ultra thin] (7) edge node[left] {} (9);
    \path [-,ultra thin] (5) edge node[left] {} (13);
    \path [-,ultra thin] (13) edge node[left] {} (14);
    \path [-,ultra thin] (6) edge node[left] {} (14);

    \path [-, thin,draw, dotted, color=black] (11) edge node[left] {} (30);
    \path [-, thin,draw, dotted, color=black] (4) edge node[left] {} (11);
    \path [-, thin,draw, dotted, color=black] (15) edge node[left] {} (7);
    \path [-, thin,draw, dotted, color=black] (15) edge node[left] {} (6);
    \path [-, thin,draw, dotted, color=black] (1) edge node[left] {} (2);



    \end{tikzpicture}
    \caption{Configuration $\ref{girth5confj}$}
    \label{configuration j}
   \end{minipage}
\end{figure}

\noindent \textbf{Configuration $(j)$.} Suppose $G$ contains configuration $(j)$. Let $v\in S(5,3)$, $v_i\in N(v)\cap S_2$ with $v_i'\in N(v_i)\backslash \{v\}$ for $i=1,2,3$, $u,w\in N(v)\cap S_{5^+}$ with $u_1,u_2\in N(u)\cap S_2$ and $w_1\in N(w)\cap S_2$ such that $\{u,w\}$ do not share a common face. Furthermore, suppose $\{v,u,u_1,v_1\}$ and $\{v,u,u_2,v_2\}$ each lie on a common $5$-face for distinct $v_1,v_2$ and the face incident to two $2$-neighbors of $v$ is a $6$-face containing three $2$-vertices, $v_1,v_3$ and $p$ where $p\in N(v_1')\cap N(v_3') $ (Figure \ref{configuration j}). Consider $B=\{w_1x|x\in N(w_1)\}\cup \{u_1x|x\in N(u_1)\}\cup \{vv_1\}$. Then $|B|=5$. Clearly, $u_1,w_1\in I'$. By Claim 1, $N(u_1)\cap I'=\emptyset$ and $N(w_1)\cap I'=\emptyset$. It follows that $v_1\in I'$. 

If $v_3\in I'$, then $p\in I'$. In particular, $p$ dominates $v_3'$. If $v_3\notin I'$, then $v_3'\in I'$. Thus, $v_3'$ is always dominated by some vertex besides $v_3$ in $I'$. Similarly, if $v_2\in I'$, then $u_2\in I'$. If $v_2\notin I'$, then $v_2'\in I'$. Thus, $v_2'$ is always dominated by some vertex besides $v_2$ in $I'$. Note that, $\{v_2,v_3\}\cap I'\neq \emptyset$ in order to dominate $v$, as $\{u,v\} \not \in I'$. Thus, $(I'\backslash N_G(v))\cup \{v\}$ is an independent dominating set of $G$ with size at most $|I'|-2+1<|I'|\leq |I|$, a contradiction.\\

Therefore, for each configuration, there exists an edge set with size at most $5$ such that its removal results in a subgraph $G'$ with strictly larger independent domination number than $G$. Hence, $b_i(G)\leq 5$ for all planar graphs with $\delta(G)\geq 2$ and $g(G)\geq 5$.

\end{proof}

\section{Graphs with $\delta(G)\geq 3$ and $g(G)\geq 4$}
\label{section3}

\begin{theorem}
\label{girth4mindegree3configurations}
Every connected planar graph G with $\delta(G) \geq 3$ and $g(G)\geq 4$ contains at least one of the following configurations:
\begin{enumerate}
   \item[(a)] \label{girth4confa} $(3,4)$-edge.
   \item[(b)] \label{girth4confb} a $5$-vertex $v$ such that $|N(v)\cap S_3|= 4$, $d(u) \leq 5$ for some $u \in N(v)\setminus S_3$, and the vertex $v$ is incident only with $4$-faces.
   \item[(c)]\label{girth4confc} a $5$-vertex $v$ with $|N(v)\cap S_3|= 5$ and incident to four $4$-faces and one $5^+$-face.
\end{enumerate}
\end{theorem}

\begin{proof}
Let $G$ be a non-trivial connected planar graph with $\delta(G) \geq 3$ and $g(G)\geq 4$ without any of the given configurations. We will proceed by balanced charging. Every face $f$ begins with charge $\ell(f)-4$, and every vertex begins with charge $d(v)-4$. Thus, only $3^-$-vertices start with an initial negative charge. We proceed with the following discharging rules:

\begin{itemize}
    \item[(R1)] Every $3$-vertex takes $ \frac{1}{3}$ from each neighbor.
     \item[(R2)] Each $5^+$-face $f$ distributes $\frac{\ell(f) - 4}{k}$ to each of its $k$ incident $5$-vertices that have at least four $3$-neighbors.
    \item[(R3)] Every $6^+$-vertices $v$ distributes $\frac{d(v) - 4-\frac{1}{3}|N(v)\cap S_3|}{k}$ to each of its $k$ $5$-neighbors that have four $3$-neighbors and incident only to $4$-faces.
\end{itemize}

We begin by verifying that all vertices end with a non-negative charge. Let $v$ be a $j$-vertex for $j\geq 3$. For $j=3$, $v$ has initial charge $3-4=-1$. It is sufficient to show that $v$ gains at least a charge of $1$. By $(R1)$, $v$ receives a charge of at least $3 (\frac{1}{3}) = 1$ from its neighbors. Thus, $3$-vertices end with a non-negative charge. For $j=4$, $v$ has initial charge $4-4=0$. By configuration $(a)$, $v$ has no $3$-neighbors. Thus, none of the discharging rules involves $4$-vertices. Hence, $4$-vertices end with a non-negative charge. For $j\geq 6$, $v$ gives $\frac{1}{3}$ to each of its $3$-neighbors by $(R1)$, and then distributes its remaining charge equally among its $k$ $5$-neighbors as specified in $(R3)$. Hence, $v$ ends with a final charge of
$$(d(v)-4)-\frac{1}{3}|N(v)\cap S_3|-\left(\frac{d(v) - 4-\frac{1}{3}|N(v)\cap S_3|}{k}\right)k=0,$$
and remains non-negative. 



Before verifying the final charge of $5$-vertices, we prove the following two facts. 
\begin{fact}
\label{6vertexcharge}
Every $6^+$-vertices  gives at least $\frac{1}{3}$ to its $5$-neighbors as described in $(R3)$.
\end{fact}

We have that $0\leq k +|N(v)\cap S_3|\leq d(v)$. As the number of $3$-neighbors increases, the number of relevant $5$-neighbors must either stay constant or decrease, increasing the amount of charge each $5$-neighbor receives overall. Thus, when $|N(v)\cap S_3|=0$, $v$ distributes at least $\frac{d(v)-4}{k}\geq \frac{6-4}{6}=\frac{1}{3}$ to each of its relevant $k$ $5$-neighbors. 

\begin{fact}
\label{5facecharge}
Every $5^+$-face $f$ gives at least $\frac{1}{3}$ to its $k$ incident $5$-vertices as described in $(R2)$.
\end{fact}
By $(R2)$, $f$ distributes charge only to its $k$ incident $5$-vertices that also have at least four $3$-neighbors. Given such a $5$-vertex, say $u$, two of its neighbors must be incident to $f$. It follows that $u$ has at least one $3$-neighbor incident to $f$. The other neighbor of $v$ that is also incident to $f$ must be either a $3$-vertex or a $4^+$-vertex, possibly another degree-seeking $5$-vertex. Hence, $f$ may only contain at most $\lceil\frac{\ell(f)}{2}\rceil$ such incident $5$-vertices; that is, $k\leq \lceil\frac{\ell(f)}{2}\rceil$. Thus when $\ell(f)\geq 5$, $f$ distributes charge at least

$$\frac{\ell(f)-4}{k}\geq \frac{\ell(f)-4}{\lceil\frac{\ell(f)}{2}\rceil}\geq \frac{1}{3}.$$

We now proceed with verifying the final charge of $5$-vertices. Let $j=d(v)=5$. If $|N(v)\cap S_3|=5$, then $v$ must have at least two incident $5^+$-faces by configuration $(c)$. By Fact \ref{5facecharge}, $v$ receives $\frac{1}{3}$ from each $5^+$-face. Thus, $v$ ends with final charge at least $(5-4)-5\left(\frac{1}{3}\right)+2\left(\frac{1}{3}\right)=0$.

If $|N(v)\cap S_3|=4$, then $v$ must have either at least one incident $5^+$-face or a $6+$-neighbor by configuration $(b)$. By Facts \ref{5facecharge} and \ref{6vertexcharge}, $v$ will receive at least $\frac{1}{3}$ from either structure. Thus, $v$ ends with a final charge of at least $(5-4)-4\left(\frac{1}{3}\right)+\frac{1}{3}=0$.

If $|N(v)\cap S_3|\leq 3$, then $v$ ends with final charge $(5-4)-3\left(\frac{1}{3}\right)=0$. Thus, all $5$-vertices, and consequently all vertices, end with a nonnegative charge. 

By our discharging rule $(R2)$, each face $f$ of $G$ takes its initial charge and distributes it equally among its $k$ incident $5$-vertices that have at least four $3$-neighbors. As our remaining rules do not involve $f$ giving any charge, every face of $G$ ends with a nonnegative charge. 
\end{proof}

Using our new configurations, we are now able to determine an upper bound for the independent bondage number for planar graphs with minimum degree at least 3 and the girth is at least $4$.

\begin{theorem}
\label{girth4mindegree3IBNproof}
Let  $G$ be a connected planar graph with $\delta(G) \geq 3$. If $g(G) \geq 4$, then $b_{i}(G) \leq 6$.
\end{theorem}

\begin{proof}
Let $G$ be a planar graph with $\delta(G)\geq 3$ and $g(G) \geq 4$. Then $G$ must contain one of the configurations listed in Theorem \ref{girth4mindegree3configurations}. If $G$ contains configuration $(a)$, then $b_i(G)\leq 4+3-1=6$ by Theorem \ref{PriddyWei}. Hence, we may assume $G$ contains either configuration $(b)$ or configuration $(c)$. Let $v$ be a $5$-vertex with $N(v)=\{u_1,u_2,u_3,u_4,u_5\}$.\\

\noindent\textbf{Configuration $(b)$.} Suppose $G$ contains configuration $(b)$ where $3\leq d(u_5)\leq 5$ and $d(u_i)=3$ for $1\leq i \leq 4$. As $v$ is incident only to $4$-faces, there exist distinct $z_i\in N(u_i)\cap N(u_{i+1})$ for $1\leq i \leq 4$ and $z_5\in N(u_5)\cap N(u_1)$. If they exist, let $x,y\in N(u_5)\backslash \{v,z_4,z_5\} $ (Figure \ref{fig:my_label}). Consider the edge set $B=\{u_1w:w\in N(u_1)\}\cup\{u_3z_2, u_3z_3, u_5x\}$. Then $|B|\leq6$. Let $I$ and $I'$ be minimum independent dominating sets of $G$ and $G'=G-B$ respectively. We claim $|I|<|I'|$. Suppose not, that is, $|I|\geq |I'|$. It is clear that $u_1\in I'$. Hence, $N(u_1)\cap I'=\emptyset$. Otherwise,  $I'\backslash \{u_1\}$ is an independent dominating set of $G$ with size $|I'|-1<|I'|\leq |I|$, a contradiction. It follows that $u_3\in I'$. Furthermore, $u_2\in I'$ as $N(u_3)\cap I'=\emptyset$. We prove the following fact.

\begin{fact}
\label{factbondagegirth4}
$[N_{G'}(z_i)\backslash \{u_{2}\}] \cap I' \neq \emptyset$ for $1\leq i \leq 2$.
\end{fact}

Suppose not for $i=1$, that is $[N_{G'}(z_1)\backslash \{u_2\}] \cap I' = \emptyset$. Then $[I'\backslash \{u_1,u_2\}] \cup \{z_1\}$ is an independent dominating set of $G$ with size $|I'|-2+1<|I'|\leq |I|$, a contradiction. By an analogous argument, $[N_{G'}(z_2)\backslash \{u_1\}] \cap I' \neq \emptyset$.\\

Now, we continue the main proof. Suppose $z_4\in I'$. It follows that $[N_{G'}(z_3)\backslash u_4]\cap I'\neq \emptyset$ to dominate $z_3$ in $G'$ as $z_3,u_4\not\in I'$. Hence, $I'\backslash \{u_1,u_2,u_3\}\cup \{v\}$ is an independent dominating set of $G$ with size $|I'|-3+1<|I'|\leq |I|$, a contradiction. Thus, we may assume $z_4\not\in I'$. It must be that $u_4\in I'$ to dominate $u_4$ as $z_3\not\in I'$. We consider two cases:

\textbf{Case 1:} $u_5\in I'$.

Using a similar argument as seen in Fact \ref{factbondagegirth4}, we have that $[N_{G'}(z_3)\backslash u_4]\cap I'\neq \emptyset$. Consider $I''=[I'\backslash\{u_i: 1\leq i \leq 5\}]\cup \{v\}$. If $[N_{G'}(z_4)\backslash \{u_4,u_5\}]\cap I'=\emptyset$, then we add $z_4$ to $I''$. If $[N_{G'}(y)\backslash \{u_5\}]\cap I'=\emptyset$, then we add $y$ to $I''$. If $[N_{G'}(z_5)\backslash \{u_5\}]\cap I'=\emptyset$, then we add $z_5$ to $I''$. Thus, $I''$ is an independent dominating set of $G$ with size at most $|I''|\leq |I'|-5+4<|I'|\leq |I|$, a contradiction.

\textbf{Case 2:} $u_5\not\in I'$.

Then, $y\in I'$ and $[N_{G'}(z_5)\backslash \{u_4\}]\cap I\neq \emptyset$. Otherwise, $I'$ is not a dominating set. Consider $I''=[I'\backslash\{u_i: 1\leq i \leq 4\}]\cup \{v\}$. For $3\leq i \leq 4$, if $[N_{G'}(z_i)\backslash \{u_4\}\cap I'=\emptyset$, then we add $z_i$ to $I''$. Thus, $I''$ is an independent dominating set of $G$ with size at most $|I''|\leq |I'|-4+3<|I'|\leq |I|$, a contradiction.

Therefore, if $G$ contains configuration $(b)$, $b_i(G)\leq 6$.\\

\noindent \textbf{Configuration $(c)$.} Suppose $G$ contains configuration $(c)$ where $d(u_i)=3$ for $1\leq i \leq 5$. Let $f_i$ be the $4$-face incident to $\{v, u_i, z_i,u_{i+1}\}$ for $1\leq i \leq 4$ and $f_5$ be the remaining $5^+$-face containing $\{v,u_1,x,y,u_5\}$ where $x\in N(u_1)\cap v(f_5)$ and $y\in N(u_5)\cap v(f_5)$ (Figure \ref{fig:my_label5cycle}). 

Consider the edge set $B=\{u_1w:w\in N(u_1)\}\cup\{u_iz_i: 2\leq i \leq 4\}$. Then $|B|=6$. Let $I$ and $I'$ be minimum independent dominating sets of $G$ and $G'=G-B$ respectively. We claim $|I|<|I'|$. Suppose not, that is, $|I|\geq |I'|$. It is clear that $u_1\in I'$. It follows that $u_2, u_3, u_4\in I'$. Using the same argument as seen in Fact \ref{factbondagegirth4}, we have $[N_{G'}(z_i)\backslash\{u_{i+1}\}]\cap I'\neq \emptyset$ for $1\leq i \leq 3$. We consider two cases:

\textbf{Case 1: } $u_5\in I'$

Consider $I''=[I'\backslash\{u_i: 1\leq i \leq 5\}]\cup \{v\}$. If $[N_{G'}(z_4)\backslash \{u_5\}]\cap I'=\emptyset$, we add $z_4$ to $I''$. Similarly, if $[N_{G'}(y)\backslash \{u_5\}]\cap I'=\emptyset$, we add $y$ to $I''$. It follows that $I''$ is an independent dominating set of $G$ with size at most $|I''|\leq |I'|-5+2<|I'|\leq |I|$, a contradiction.

\textbf{Case 2: } $u_5\not\in I'$.

Then it must be that $y\in I'$. It follows that $N_{G'}(z_4)\cap I'\neq\emptyset$ as $I'$ is dominating set. Hence, $[I'\backslash\{u_i: 1\leq i \leq 4\}]\cup \{v\}$ is an independent dominating set of $G$ with size $|I'|-4+1<|I'|\leq |I|$, a contradiction.

Thus, if $G$ contains configuration $(c)$, $b_i(G)\leq 6$, completing all listed configurations of $G$.
\end{proof}

\begin{figure}[!htb]
    \centering
    \begin{tikzpicture}[scale=.7]
   [
   > = stealth, 
    shorten > = 1pt, 
    auto,
    node distance = 2cm, 
    thick, dashed pattern=on 
    ]

    \tikzset{every state}=[
    draw = black,
    thick,
    fill = white,
    minimum size = 1mm
    ]
    \node[circle,draw=black, fill=black,scale=0.5, label = below:\small{$v$}](1) at (0,0){};
    \path (1) ++(0:2) node[circle,draw=black, fill=black,scale=0.5, label = right:\small{$u_1$}](2) {};
    \path (1) ++(72:2) node[circle,draw=black, fill=black,scale=0.5, label = above:\small{$u_2$}](3) {};
    \path (1) ++(144:2) node[circle,draw=black, fill=black,scale=0.5, label = above left:\small{$u_3$}](4) {};
    \path (1) ++(216:2) node[circle,draw=black, fill=black,scale=0.5, label = left:\small{$u_4$}](5) {};
    \path (1) ++(288:2) node[circle,draw=black, fill=black,scale=0.5, label = above right:\small{$u_5$}](6) {};
    \path (2) ++(72:2) node[circle,draw=black, fill=black,scale=0.5, label = below right:\small{$z_1$}](7) {};
    \path (3) ++(144:2) node[circle,draw=black, fill=black,scale=0.5, label = right:\small{$z_2$}](8) {};
    \path (4) ++(216:2) node[circle,draw=black, fill=black,scale=0.5, label = above:\small{$z_3$}](9) {};
    \path (5) ++(288:2) node[circle,draw=black, fill=black,scale=0.5, label = left:\small{$z_4$}](10) {};
    \path (6) ++(360:2) node[circle,draw=black, fill=black,scale=0.5, label = right:\small{$z_5$}](11) {};
    \path (6) ++(275:2) node[circle,draw=black, fill=black,scale=0.5, label = below:\small{$x$}](12) {};
    \path (6) ++(300:2) node[circle,draw=black, fill=black,scale=0.5, label = below:\small{$y$}](13) {};

    \path [-, ultra thick ,draw, dotted, color=black] (1) edge node[left] {} (2);
    \path [-,ultra thin] (1) edge node[left] {} (3);
    \path [-,ultra thin] (1) edge node[left] {} (4);
    \path [-,ultra thin] (1) edge node[left] {} (5);
    \path [-,ultra thin] (1) edge node[left] {} (6);
    \path [-, ultra thick ,draw, dotted, color=black] (2) edge node[left] {} (7);
    \path [-,ultra thin] (3) edge node[left] {} (7);
    \path [-,ultra thin] (3) edge node[left] {} (8);
    \path [-, ultra thick ,draw, dotted, color=black] (4) edge node[left] {} (8);
    \path [-, ultra thick ,draw, dotted, color=black] (4) edge node[left] {} (9);
    \path [-,ultra thin] (5) edge node[left] {} (9);
    \path [-,ultra thin] (5) edge node[left] {} (10);
     \path [-,ultra thin] (6) edge node[left] {} (10);
    \path [-,ultra thin] (6) edge node[left] {} (11);
     \path [-, ultra thick ,draw, dotted, color=black] (2) edge node[left] {} (11);
     
     \path [-, ultra thick ,draw, dotted, color=black] (6) edge node[left] {} (12);
     \path [-, thin,draw, dotted, color=black] (6) edge node[left] {} (13);

    \end{tikzpicture}
    \caption{A $5$-vertex $v$ adjacent to four $3$-vertices and one $5^-$-vertex, with $v$ incident to five $4$-faces}

    \label{fig:my_label}
\end{figure}

\begin{figure}[!htb]
    \centering
    \begin{tikzpicture}[scale=.7]
   [
    > = stealth, 
    shorten > = 1pt, 
    auto,
    node distance = 2cm, 
    thick, dashed pattern=on 
    ]

    \tikzset{every state}=[
    draw = black,
    thick,
    fill = white,
    minimum size = 1mm
    ]
    \node[circle,draw=black, fill=black,scale=0.5, label = below:\small{$v$}](1) at (0,0){};
    \path (1) ++(0:2) node[circle,draw=black, fill=black,scale=0.5, label = right:\small{$u_1$}](2) {};
    \path (1) ++(72:2) node[circle,draw=black, fill=black,scale=0.5, label = above:\small{$u_2$}](3) {};
    \path (1) ++(144:2) node[circle,draw=black, fill=black,scale=0.5, label = above left:\small{$u_3$}](4) {};
    \path (1) ++(216:2) node[circle,draw=black, fill=black,scale=0.5, label = left:\small{$u_4$}](5) {};
    \path (1) ++(288:2) node[circle,draw=black, fill=black,scale=0.5, label = above right:\small{$u_5$}](6) {};
    \path (2) ++(72:2) node[circle,draw=black, fill=black,scale=0.5, label = below right:\small{$z_1$}](7) {};
    \path (3) ++(144:2) node[circle,draw=black, fill=black,scale=0.5, label = right:\small{$z_2$}](8) {};
    \path (4) ++(216:2) node[circle,draw=black, fill=black,scale=0.5, label = above:\small{$z_3$}](9) {};
    \path (5) ++(288:2) node[circle,draw=black, fill=black,scale=0.5, label = left:\small{$z_4$}](10) {};
    \path (6) ++(360:2) node[circle,draw=black, fill=black,scale=0.5, label = right:\small{$x$}](11) {};
    \path (6) ++(300:2) node[circle,draw=black, fill=black,scale=0.5, label = below:\small{$y$}](13) {};

    \path [-, ultra thick ,draw, dotted, color=black] (1) edge node[left] {} (2);
    \path [-,ultra thin] (1) edge node[left] {} (3);
    \path [-,ultra thin] (1) edge node[left] {} (4);
    \path [-,ultra thin] (1) edge node[left] {} (5);
    \path [-,ultra thin] (1) edge node[left] {} (6);
   
    \path [-,ultra thin] (3) edge node[left] {} (7);
    \path [-, ultra thick ,draw, dotted, color=black] (3) edge node[left] {} (8);
    \path [-,ultra thin] (4) edge node[left] {} (8);
  
    \path [-,ultra thin] (5) edge node[left] {} (9);
    \path [-, ultra thick ,draw, dotted, color=black] (5) edge node[left] {} (10);
    \path [-,ultra thin] (6) edge node[left] {} (10); 
     \path [-, ultra thick ,draw, dotted, color=black] (2) edge node[left] {} (11);
     \path [-,ultra thin] (6) edge node[left] {} (13);

    \path [-, ultra thick ,draw, dotted, color=black] (2) edge node[left] {} (7);
    
    \path [-, ultra thick ,draw, dotted, color=black] (4) edge node[left] {} (9);


     \path [-, thin,draw, dotted, color=black] (11) edge node[left] {} (13);

    \end{tikzpicture}
    \caption{A $5$-vertex $v$ adjacent to five $3$-vertices with $v$ incident to four $4$-faces and one $5^+$-face (dashed link between $x,y$ is a path of length at least one)}
    \label{fig:my_label5cycle}
\end{figure}

\section{Graphs with $\delta(G)\geq 2$ and $g(G)\geq 7$}
\label{section4}
In \cite{cranston2017introduction}, Cranston and West prove the following lemma.
\begin{lemma}
    Every planar graph $G$ with girth at least $7$ and $\delta(G)\geq 2$ has a $2$-vertex with a $3^-$-neighbor.
\end{lemma}
From this lemma, we obtain the following corollary.
\begin{corollary}
\label{IBNgirth7}
A planar graph $G$ with $\delta(G)\geq 2$ and $ g(G)\geq 7$ has $b_i(G)\leq 4$.
\end{corollary}

Although the following result is not required for the proof of the independent bondage number (since we used a different lemma from another source), we include it here as it yields some interesting configurations.

\begin{theorem}
\label{confgirth7}
Every connected planar graph $G$ with $\delta(G)\geq2$ and $g(G)\geq 7$ contains at least one of the following configurations:
\begin{itemize}
    \item[(a)] a $(2,3)$-edge.
    \item[(b)] a $4^+$-vertex $v$ with $[d(v)-1]$ $2$-neighbors and one $3^-$-neighbor.
\end{itemize}
\end{theorem}
\begin{proof}
Let $G$ be a planar graph with $\delta(G)\geq 2$ and $g(G)\geq 7$. Suppose $G$ contains none of the listed configurations. We will proceed with face charging and assign an initial charge of $2d(v)-6$ to each vertex $v$ and $\ell(f)-6$ to each face $f$. We redistribute the charge according to the following discharging rule.
\begin{itemize}
    \item[(R1)] Every $2$-vertex takes $\frac{2d(u)-6}{d(u)-1}$ from each neighbor $u$ and $\frac{\ell(f)-6}{\big\lfloor\frac{\ell(f)}{2}\big\rfloor}$ from each incident face $f$.
\end{itemize}
We begin by showing all vertices end with a nonnegative final charge. We note that if $2$-vertex is incident to only one distinct face, then the $2$-vertex takes charge twice from its incident face. However, since the vertex is then counted twice in the length of the face, the final charges does not change in our discharging rule. Hence, we perform our discharging rule as if each vertex has distinct faces. 

Let $v$ be a $j$-vertex. For $j=2$,  $v$ must have only $4^+$-neighbors by configuration $(a)$. Furthermore by $(R1)$, $v$ takes charge at least $\frac{2(4)-6}{4-1}=\frac{2}{3}$ from each $4^+$-neighbor. As $g(G)\geq 7$, $v$ takes charge at least $\frac{7-6}{\big\lfloor \frac{7}{2}\big\rfloor}=\frac{1}{3}$ from each of its incident faces. Thus, $v$ ends with a final charge of at least $[2(2)-6]+2\left(\frac{2}{3}\right)+2\left(\frac{1}{3}\right)=0$.

For $j=3$, $v$ has initial charge $2(3)-6=0$. As $v$ does not have any $2$-neighbors by configuration $(a)$, then $v$ does not lose any charge to its neighbors. It follows that $v$ begins and ends with a nonnegative final charge. 

For $j=4$, $v$ has at most $[d(v)-1]$ $2$-neighbors by configuration $(b)$. By $(R1)$, $v$ distributes its initial positive charge equally among its $2$-neighbors. Thus, $v$ does not lose too much charge. It follows that all vertices end with a non-negative charge.

We now consider the faces of $G$. By configuration $(a)$, every face $f$ must be incident to at most $\big\lfloor\frac{l(f)}{2}\big\rfloor$ $2$-vertices. By $(R1)$, each face distributes its initial positive charge equally among its incident $2$-vertices. Thus, every face does not lose too much charge and has a non-negative final charge. Hence, all faces end with a non-negative charge. 
\end{proof}

\section{Graphs with $\delta(G)\geq 2$ and $g(G)\geq 10$}
\label{section5}
\begin{theorem}
    \label{girth10discharging}
     Let $G$ be a planar graph with $\delta(G)\geq 2$ and $g(G)\geq 10$. Then  $G$ has at least one of the following configurations.
\begin{itemize}
    \item[(a)] a $(2,2)$-edge;
    \item[(b)] a vertex $v\in S(3^+, d(v))$.
\end{itemize}   
\end{theorem}

\begin{proof}
Let $G$ be a planar graph with $\delta(G)\geq 2$ and $g(G)\geq 10$. Suppose $G$ contains neither of the listed configurations. We will proceed by face charging and assign an initial charge of $2d(v)-6$ to each vertex $v$ and $\ell(f)-6$ to each face $f$. We redistribute the charge according to the following discharging rules.
\begin{itemize}
    \item[(R1)] Every $2$-vertex $v$ takes $\frac{1}{5}$ from each neighbor and $\frac{8}{5k}$ from its distinct $k$ incident faces.
    \item[(R2)] Every vertex $v\in S(3,t)$ where $0\leq t \leq d(v)-1$ takes $\frac{t}{5k}$ from each of its $k$ distinct incident faces that contain a $3^+$-neighbor of $v$.
\end{itemize} 
We begin by showing all vertices end with a nonnegative final charge. Let $v$ be a $j$-vertex. For $j=2$, $v$ must be adjacent to only $3^+$-vertices by configuration $(a)$. Furthermore by our discharging rules, $2$-vertices do not give any charge. By rule $(R1)$, $v$ take $\frac{1}{5}$ from each neighbor and $\frac{8}{5k}$ from its distinct $k$ incident faces. Hence, $v$ ends with final charge at least $(2(2)-6)+2\left(\frac{1}{5}\right)+\frac{8}{5}=0.$

For $j=3$, $v$ has an initial charge of $2(3)-6=0$. By $(R1)$, $v$ loses charge only to its $2$-neighbors. For $v\in S(3,t)$, we have that $0\leq t \leq 2$ by configuration $(b)$. It follows that $v$ loses a total charge of $t\left(\frac{1}{5}\right)=\frac{t}{5}$ to its $2$-neighbors. As $t=|N(v)\cap S_2|\leq 2$, then $v$ must always be incident to at least one face containing a $(3,3^+)$-edge. Thus, by $(R2)$, $v$ gains a total charge of $\left(\frac{t}{5k}\right)k=\frac{t}{5}$ from its $k$ distinct incident faces  that also contain a $3^+$-neighbor of $v$. Thus, $v$ begins and ends with a final charge of $0$. Hence, all $3$-vertices end with a nonnegative final charge.

For $j\geq 4$, $v$ loses charge only its $2$-neighbors. By configuration $(b)$, $|N(v)\cap S_2|\leq j-1$. As $v$ has at most $j-1$ $2$-neighbors, then $v$ ends with a final charge of at least $(2j-6)-(j-1)\left(\frac{1}{5}\right)=\frac{9j-29}5\geq 0$ when $j\geq 4$. Thus, all vertices end with a nonnegative charge.

We now consider the faces of $G$. Let $f$ be an $\ell$-face for $\ell\geq 10$. If a $2$-vertex $v$ has only one distinct face, then we say that $v$ takes a charge of $\frac{4}{5}$ twice from $f$. Similarly, we do the same for degree-seeking $3$-vertices. As the charge taken is the same, we may assume that $f$ has $\ell(f)$ distinct incident vertices. By our discharging rules, $f$ loses $\frac{4 }{5}$ for each vertex in $V(f)\cap S_2$, $\frac{1}{15}$ for each vertex in $V(f)\cap S(3,1)$ with $3^+$-neighbor incident to $f$, and $\frac{2}{15}$ for each vertex in $V(f)\cap S(3,2)$ with $3^+$-neighbor incident to $f$. Thus, $f$ loses charge equal to $\frac{4}{5}|V(f)\cap S_2|+\sum_{v\in V(f)\cap S(3,t),t\in{1,2}}\frac{t}{15}$. 

Let $(u_1,u_2,u_3)$ be a $3$-path along $f$ such that $d(u_2)\geq 3$. If $d(u_3)\geq 3$, then $u_2$ and $u_3$ take charge at most $2\left(\frac{2}{15}\right)=\frac{4}{15}$ from $f$. If $d(u_3)=2$, then $u_2$ and $u_3$ take charge at most $\frac{2}{15}+\frac{4}{5}=\frac{14}{15}>\frac{4}{15}$ from $f$. Hence, the vertices of a $(3^+,2)$-edge will always take more charge from $f$ than the vertices of a $(3^+,3^+)$-edge. As our discharging rules depend only on the neighbors and incident faces of each vertex, then it suffices to only consider the case such that each $3^+$-vertex is followed by a $2$-vertex along $f$. 

If $d(u_1)=d(u_3)=2$, then $u_1$, $u_2$, and $u_3$ take a charge of $2\left(\frac{4}{5}\right)=\frac{8}{5}$. If $d(u_1)\geq 3$ and $d(u_3)=2$, then $u_1$, $u_2$, and $u_3$ take a charge of at most $\frac{4}{5}+2\left(\frac{2}{15}\right)=\frac{16}{15}<\frac{8}{5}$. Hence, any configuration of the vertices of $f$ that does not maximize its number of incident $2$-vertices will always take strictly less charge in comparison to the configuration with the maximum number of incident $2$-vertices. 

Therefore, it suffices to consider only the case where $f$ contains the maximum $\lfloor\frac{\ell}{2}\rfloor$ $2$-vertices. If $\ell(f)$ is even, then the only degree-seeking vertices of $f$ are $2$-vertices. Thus, $f$ ends with final charge at least $(\ell(f)-6)-\lfloor\frac{\ell}{2}\rfloor\left( \frac{4}{5}\right)\geq 0$ when $\ell(f)\geq 10$. If $\ell(f)$ is odd, then $f$ contains at most two degree-seeking vertices that belong to $S(3,1)\cup S(3,2)$ (rule $(R2)$). It follows that $f$ ends with final charge at least $(\ell(f)-6)-\lfloor\frac{\ell}{2}\rfloor\left( \frac{4}{5}\right)-2\left(\frac{2}{15}\right)\geq 0$ when $\ell(f)\geq 10$. 
\end{proof}

\begin{theorem}
\label{confgirth10}
A planar graph $G$ with $\delta(G)\geq 2$ and $g(G)\geq 10$ has $b_i(G)\leq3$.
\end{theorem}
\begin{proof}
Let $G$ be a planar graph with minimum degree $2$ and girth at least $10$. It follows that $G$ contains at least one of the listed configurations listed in Theorem \ref{girth10discharging}. If $G$ contains Configuration $(a)$, then $b_i(G)\leq 2+2-1=3$ by Theorem \ref{PriddyWei}. Hence, we may assume $G$ contains only configuration $(b)$. Let $v\in S(3^+, d(v))$ with $N(v)=\{u_1,u_2,\ldots, u_{d(v)}\}$. Consider the edge set $B=\{u_1x : x\in N(u_1)\}\cup \{u_2x: x\in N(u_2)\backslash\{v\}\}$.

Then $|B|=3$. Let $I$ and $I'$ be minimum independent dominating sets of $G$ and $G'=G-B$ respectively. We claim $|I|<|I'|$. Suppose not, that is $|I|\geq |I'|$. It is clear that $u_1\in I'$. It follows that $v\notin I'$. Otherwise, $I'\backslash\{u_1\}$ is an independent dominating set of $G$ with size $|I'|-1<|I'|\leq |I|$, a contradiction as $I$ is a minimum independent dominating set of $G$. Furthermore, it must be that $u_2\in I'$ since $u_2$ has no other neighbors besides $v$ in $G'$.

As $d(u_t)=2$ for $1\leq t \leq d(v)$, then either $u_t\in I'$ or there exists a $u_t'\in [N(u_t)\setminus \{v\}]\cap I'$ since $v\not\in I'$. If $u_t'\in I'$ for all $3\leq t\leq d(v)$, then $I'\backslash \{u_1,u_2\}\cup\{v\}$ is an independent dominating set of $G$ of size $|I'|-2+1<|I'|\leq |I|$, a contradiction. Thus, we may assume that at least one $u_t\in I'$. 

Fix $3\leq i\leq d(v)$ such that $u_i\in I'$. Either $|N_{G'}[u_i']\cap I'|=1$ or $|N_{G'}[u_i']\cap I'|\geq 2$. If $|N_{G'}[u_i']\cap I'|=1$, then $N_{G'}(u_i')\cap I'=\{u_i\}$. It follows that $[I'\backslash\{u_i\}]\cup\{u_i'\}$ is a minimum independent dominating set of $G'$. Thus, if $|N_{G'}[u_t']\cap I'|=1$ for some $t$, then we may assume $u_t'\in I'$.

Let $U$ be the set of all $u_t$ such that $|N_{G'}[u_t']\cap I'|\geq 2$. Then for each $u_t\in U$, there exists another vertex that dominates $u_t'$. Hence, $I'\backslash (U\cup\{u_1, u_2\})\cup \{v\}$ is an independent dominating set of $G$ with size $|I'|-|U|-2+1<|I'|\leq |I|$, a contradiction. Hence, it must be that $|I|<|I'|$. Therefore, $B$ is an independent bondage set and $b_i(G)\leq 3$.
\end{proof}

\section{Conclusion and Future Work}
\label{conclusionsection}
We now summarize our findings in the following theorem.
\begin{theorem}
For any connected planar graph $G$,
\[
b_i(G)\leq 
\begin{cases}
    6, & \text{if } g(G)\geq 4 \text{ and } \delta(G)\geq 3, \\
    5, & \text{if } g(G)\geq 5 \text{ and } \delta(G)\geq 2, \\
    4, & \text{if } g(G)\geq 7 \text{ and } \delta(G)\geq 2, \\
    3, & \text{if } g(G)\geq 10 \text{ and } \delta(G)\geq 2.
\end{cases}
\]
\end{theorem}

In this study, we established new upper bounds on the independent bondage number of planar graphs under various girth and minimum degree constraints. Our results extend earlier work on the bondage number of graphs and demonstrate that structural constraints of planar graphs can significantly restrict the independent bondage number. 

These findings provide the first systematic results on the independent bondage number under girth restrictions and open several directions for further research. One natural direction is to investigate whether the bounds obtained here are tight, either by constructing extremal examples or proving sharper results. Another promising avenue is to study the independent bondage number in other graph classes, such as toroidal graphs or graphs with bounded maximum degree.

\end{document}